\documentclass[11pt]{article}

\usepackage{amsmath}
\usepackage{amsthm}
\usepackage{mathtools}
\usepackage{booktabs}
\usepackage{dsfont}
\usepackage{url}

\usepackage{ifxetex,ifluatex}
\ifxetex
\usepackage{xltxtra}
\else
\ifluatex
\usepackage{luatextra}
\else
\usepackage{amssymb}
\usepackage[utf8]{inputenc}
\usepackage[T1]{fontenc}
\usepackage{mathpazo,tgcursor,tgpagella}
\usepackage{textcomp}
\usepackage{supertabular}
\usepackage[footnotesize,nooneline,bf]{caption}
\usepackage{algorithm}
\usepackage{algpseudocode}

\usepackage{pdflscape}
\usepackage{afterpage}

\usepackage{framed,tocloft,multicol,tikz}
\usetikzlibrary{shapes,arrows}

\csname else\expandafter\endcsname
\romannumeral -`0
\fi
\fi
\iftrue\relax%
\defaultfontfeatures{Ligatures=TeX}
\usepackage[math-style=TeX, vargreek-shape=TeX]{unicode-math}
\fi

\usepackage[margin=1in]{geometry}

\usepackage{sistyle}
\SIthousandsep{,}

\usepackage{graphicx,color,authblk}
\usepackage[authoryear]{natbib}
\usepackage[compact]{titlesec} 
\usepackage{enumitem}

\usepackage{pdfpages}

\usepackage[bookmarksnumbered,bookmarksopen,plainpages=false,pdfpagelabels,%
breaklinks,colorlinks,citecolor=blue,linkcolor=blue,hyperindex]{hyperref}

\hypersetup{
  pdftitle={Solving MIPs via Scaling-based Augmentation}
  pdfauthor={Pierre Le Bodic, Jeffrey W. Pavelka, Marc E. Pfetsch, Sebastian Pokutta},
  pdfkeywords={},
  pdfsubject={MSC2000 Primary ; Secondary}
}

\newcommand{\R}{\mathds{R}}
\newcommand{\Z}{\mathds{Z}}
\newcommand{\Q}{\mathds{Q}}
\newcommand{\N}{\mathds{N}}
\newcommand{\ones}{\mathds{1}}

\DeclareMathOperator{\supp}{supp} 
\DeclareMathOperator{\poly}{poly}
\DeclareMathOperator{\sgn}{sgn}

\DeclareMathOperator*{\argmax}{\arg\!\max}

\newcommand{\cube}[1]{[0,1]^{#1}}
\newcommand{\face}[1]{\left\{#1\right\}}
\newcommand{\conv}[1]{\operatorname{conv}\left(#1\right)}
\newcommand{\card}[1]{\left\lvert#1\right\rvert}
\newcommand{\onorm}[1]{\left\lVert#1\right\rVert_1}

\newcommand{\mnorm}[1]{\left\lVert#1\right\rVert_\infty}
\newcommand{\supf}[1]{\supp(#1)}
\newcommand{\define}{\coloneqq}
\newcommand{\enifed}{\eqqcolon}
\newcommand{\order}[1]{O({#1})}
\newcommand{\binSet}{\{0,1\}}

\newcommand{\floor}[1]{\lfloor{#1}\rfloor}
\newcommand{\ceil}[1]{\lceil{#1}\rceil}
\newcommand{\abs}[1]{\lvert{#1}\rvert}

\newcommand*{\set}[2]{\left\{{#1}\,\middle|\,{#2}\right\}}

\newtheorem{theorem}{Theorem}[section]
\newtheorem{lemma}[theorem]{Lemma}

\newtheorem{observation}[theorem]{Observation}
\newtheorem{corollary}[theorem]{Corollary}

\theoremstyle{definition}

\newtheorem{defn}[theorem]{Definition}

\theoremstyle{remark}

\newtheorem{rem}[theorem]{Remark}

\title{Solving MIPs via Scaling-based Augmentation}
\author[1]{Pierre Le Bodic}
\affil[1]{ISyE,
  Georgia Institute of Technology,
  Atlanta, GA 30332, USA.
  \textit{Email:}~lebodic@gatech.edu}

\author[2]{Jeffrey W. Pavelka}
\affil[2]{ISyE,
  Georgia Institute of Technology,
  Atlanta, GA 30332, USA.
  \textit{Email:}~jpavelka@gatech.edu}

\author[3]{Marc E. Pfetsch}
\affil[3]{Department of Mathematics, TU Darmstadt, Germany.
  \textit{Email:}~pfetsch@opt.tu-darmstadt.de}

\author[4]{Sebastian Pokutta}
\affil[4]{ISyE,
  Georgia Institute of Technology,
  Atlanta, GA 30332, USA.
  \textit{Email:}~sebastian.pokutta@isye.gatech.edu}

\begin{document}

\maketitle

\begin{abstract}
  \noindent
  Augmentation methods for mixed-integer (linear) programs are a class of
  primal solution approaches in which a current iterate is augmented to a
  better solution or proved optimal. It is well known
  that the performance of these methods, i.e., number of iterations needed,
  can theoretically be improved by scaling methods. We extend these results
  by an improved and extended convergence analysis, which shows that bit
  scaling and geometric scaling theoretically perform similarly well
 in the worst case for 0/1 polytopes as well as show that in some
 cases geometric scaling can outperform bit scaling arbitrarily. 

 We also investigate the performance of implementations of these
 methods, where the augmentation directions are computed by a MIP
 solver. It turns out that the number of iterations is usually
 low. While scaling methods usually do not improve the performance for
 easier problems, in the case of hard mixed-integer optimization
 problems they allow to compute solutions of very good quality and are
 often superior.
\end{abstract}

\section{Introduction}
\label{sec:introduction}

A standard approach to solving mixed integer linear programs (MIPs) is
via a combination of branch-and-bound and cutting planes. This
approach can be considered largely as \emph{dual}, since in practice the nodes of the 
branch-and-bound tree are, to a large extent, pruned by objective bounds. An alternative
view to solving MIPs is via \emph{primal augmentation approaches}. Here the
idea is to start from a feasible solution to the considered MIP and
then move to a new solution with improved objective function value by
means of an \emph{augmentation step}.

In this work, we will consider a specific class of primal augmentation
approaches, namely those arising from \emph{scaling}. In a nutshell,
here the objective function is adjusted to include a \emph{potential}
that guides the search away from the boundary, deep into the
feasible regions similar to interior point methods in convex
optimization. In the same vein, a scaling parameter \(\mu\)
controls the tradeoff between depth in the feasible region and
optimizing the objective function. A key insight is that via
appropriate scaling only a polynomial number of
augmentation steps is needed. If now the computation of an augmenting direction
can be performed fast, one can obtain, \emph{in theory}, fast
algorithms for solving MIPs. For example, this augmentation can be
performed very fast for network flows, which, in fact, motivated several
scaling approaches for MIPs in the first
place (see e.g., \cite{wallacher1992combinatorial},
\cite{orlin1992new}). Traditional scaling approaches in the context
of network flows include the well-known capacity scaling
(scaling in the dual) and cost scaling (scaling in the primal).

While our focus is on the \emph{computational feasibility and performance} of these
scaling approaches for MIPs, we also provide new theoretical insights in
terms of worst-case examples for bit scaling, and we slightly improve
the analysis of geometric scaling. 

\subsection{Related work}
\label{sec:related-work}

Primal augmentation approaches in the context of MIPs have been
well-studied, both from an algebraic point of view using test sets, but also in the context of solving linear programs
and mixed-integer (nonlinear) programs exactly and approximately. 
\cite{Gra75} studied test sets (or Graver bases), i.e., the sets of feasible
(integer) directions, which give rise to a natural converging augmentation
algorithm; see also \cite{Sca97}. Algebraic approaches (see
e.g., \cite{de2013algebraic,de2014augmentation} and the references
contained therein) are usually based on an algebraic characterization of
test sets; then an improving direction is used for augmentation.

Augmentation methods have recently become important to investigate
mixed-integer nonlinear problems (MINLPs), see, e.g., \cite{HemLW10} and
\cite{Onn10} for an overview. Here, test sets are used to solve or
approximate MINLPs; some selected references are
\cite{HemKW14,HemOW11,LeeORW12,LeeOW08,LoeHKW08}. However, in this paper we
concentrate on mixed-integer linear programs.

In \cite{bienstock1999approximately,bienstock2002potential}, among
other approaches, an exponential
penalty function framework is considered for (approximately) solving
linear programs. Interestingly, this approach can be considered
somewhat dual to the approximate LP solving framework via
multiplicative weight updates in \cite{plotkin1995fast} for fractional
packing and covering problems (see also
\cite{arora2012multiplicative}). In \cite{letchford2003augment}, the
authors consider an integrated augment-and-branch-and-cut
framework for mixed 0/1 programs. A proximity search heuristic is considered in
\cite{fischetti2014proximity}, where the objective function is
replaced by a proximity function to explore the neighborhood around a
feasible solution.

Our approach here is mostly based on \emph{geometric scaling} introduced in \cite{schulz2002complexity}, which in
turn is inspired by classical scaling algorithms for flow problems and
certain linear programs (see e.g., \cite{wallacher1992combinatorial},
\cite{orlin1992new}, \cite{mccormick2000minimum}), as well as \emph{bit
scaling} introduced in \cite{schulz19950}, which is based on an article by
\cite{edmonds1972theoretical}. Other approaches that use scaling
implicitly are the \emph{multiplicative weights update method}
(see e.g., \cite{arora2012multiplicative}), which is also at the core
of the algorithm in \cite{garg2007faster} for
multicommodity flows. 

On a high level, the augmentation methods considered here are similar to
proximal methods for nonlinear programs (see, e.g., \cite{Rockafellar1976})
in the sense that the deviation from the current iterate is penalized in
the objective function; this is also the viewpoint of
\cite{fischetti2014proximity}, mentioned above. On the other hand, local branching,
see~\cite{FisL03}, would be the analogue of trust region methods, see,
e.g., \cite{ConnGouldToint2000}.

\subsection{Contributions}
\label{sec:contributions}

Our contributions fall into two main categories:
\begin{enumerate}
\item \emph{Theoretical analysis of primal scaling approaches.} In the
  first part we revisit bit scaling and geometric scaling. We
  establish a new upper bound on the number of required augmentations for
  geometric scaling (Theorem~\ref{thm:mraScalingRt}), which improves
  over the bound of \cite{schulz2002complexity} by a \(\log n\)
  factor, and we derive an alternative variant of geometric scaling in
  the 0/1 case that does not require the describing system to be in equality
  form (Theorem~\ref{thm:geoScalGen01}). As a consequence, this shows that
  geometric scaling is (at least) as versatile as bit scaling, since for 0/1 polytopes
  geometric scaling is no worse than bit 
  scaling (Corollaries~\ref{cor:worstCaseIdentical} and
  \ref{cor:identicalWorstCaseGen01}). We also establish a
  simple improvement for bit scaling and geometric scaling over 0/1 polytopes, whenever
  a certain \emph{width} (number of nonzero entries in any integral
  solution) is low (Theorem~\ref{thm.limitedNonzeros}). We then continue to
  show that bit scaling can be arbitrarily worse compared to
  geometric scaling, by providing an example where \(O(n)\)
  augmentations are sufficient for geometric scaling, but bit scaling can require an
  arbitrary number of augmentations (Section~\ref{sec:worst-case-examples}).
  Moreover, the number of bit scaling augmentations meets the
  theoretical upper bound up to a constant factor.

\item \emph{Computational tests of various variants of scaling.} In the
  second part, we compare implementations of bit scaling, maximum-ratio
  augmentation (MRA), and geometric scaling. Additionally, we implemented a
  primal heuristic based on geometric scaling and a straightforward
  augmentation method that simply checks for an improving solution (see
  Section~\ref{sec:implementation}). The computations are performed on
  three different testsets. The results show that the augmentation methods
  use surprisingly few iterations. While MRA is relatively slow, bit
  scaling and geometric scaling perform well, but the application of bit
  scaling is limited to instances with different objective coefficients. It
  also turns out that the augmentation methods do not seem to be helpful
  for instances that can be solved in reasonable time, e.g., on MIPLIB 2010
  benchmark instances. However, geometric scaling helps to find primal
  solutions of very good quality for very hard instances and outperforms
  the default settings. This advantage also carries over to the primal
  heuristic based on geometric scaling, which also performs very well.
\end{enumerate}

\subsection{Outline}
\label{sec:outline}

In Section~\ref{sec:preliminaries} we provide a brief summary of our
notation and preliminaries. In Section~\ref{sec:scaling-techniques} we
then consider bit scaling and geometric scaling, review known results,
and provide various improvements and comparisons. We then provide a
worst-case example for bit scaling in
Section~\ref{sec:worst-case-examples}, showing that geometric scaling
can outperform bit scaling by an arbitrary factor.  In
Section~\ref{sec:implementation} we discuss our implementations and
in Section~\ref{sec:comp-results} provide a comprehensive set of
computational results.

\section{Preliminaries}
\label{sec:preliminaries}

Our goal is to solve
\[
\max \set{cx}{A x = b,\; l \leq x \leq u,\; x \in \Z^n},
\]
where \(A \in \R^{m \times n}\), \(b \in \R^m\), and \(l\), \(u\), \(c \in
\R^n\). Note that we write \(xy\) for the inner product of two vectors \(x\), \(y
\in \R^n\). By assumption \(l\) and \(u\) are finite, and thus \(P \define \set{x \in
  \R^n}{Ax = b,\, l \leq x \leq u}\) is a polytope. We let \(P_I \coloneqq \conv{P \cap \Z^n}\) be
the \emph{integral hull} of~\(P\).

We will consider systems in equality form (except for the bounds), which
is important since we apply potential functions to the above form. In
Section~\ref{sec:geom-scal-arbitr}, we will relax the equality form
condition for the case that \(P \subseteq \cube{n}\) and effectively allow
for arbitrary representations.

We will often work with \emph{directions} \(x - y\) induced by two vectors
\(x, y \in \R^n\). If \(P \subseteq \R^n\) is a polyhedron, we say that \(z
\in \R^n\) is a \emph{feasible direction} for \(x \in P\) if \(x + z \in
P\). Moreover, \(z\) is an \emph{augmenting direction} if \(cz > 0\), and it
is an \emph{integer feasible direction} if \(z \in \Z^n\).

We denote by \(\ones\) the all-one vector and write \([n] \define \{1,
\dots, n\}\) for \(n \in \Z_{+}\). For a vector \(x \in \R^n\), let
\(\supf{x} \define \set{j \in [n]}{x_j \neq 0}\) be the \emph{support}
of~\(x\). All logarithms in this paper will be to
the basis~\(2\). All other notation is standard and can be found in
\cite{schrijver1986theory} and \cite{nw1988}, for example.

\section{Scaling techniques}
\label{sec:scaling-techniques}

The idea of \emph{scaling} is to replace a given optimization problem with
a sequence of \emph{augmentation problems}. The augmentation steps are
controlled by means of an appropriate objective function using a potential
function. This ensures a minimum relative progress in each augmentation and
will lead to oracle-polynomial running times using an augmentation oracle.

\subsection{Bit scaling}
\label{sec:bit-scaling}

We first present the \emph{bit scaling technique} for solving 0/1
programs (see \cite{schulz19950}; also
\cite{edmonds1972theoretical,graham1995handbook}). We want to maximize
a linear function \(c x\),
with \(c \in \Z^n\)
over \(P \cap \{0,1\}^n\)
with maximum absolute value \(\mnorm{c}\).
For the sake of exposition, and without loss of generality, we confine
ourselves to \(c \geq 0\)
by applying suitable coordinate flips \(x_i \mapsto 1-x_i\).
In the mixed-integer version of the algorithm (see
Algorithm~\ref{alg:bitScalingMip}), however, we will deal with an
arbitrary \(c \in \Z^n\).
The classical bit scaling algorithm is given in
Algorithm~\ref{alg:bitScaling}.

\begin{algorithm}
  \caption{\label{alg:bitScaling}Bit scaling} 
  \textbf{Input: } Feasible solution \(x^0\)\\
  \textbf{Output: } Optimal solution of \(\max \set{cx}{x \in P \cap \Z^n}\)
  \begin{algorithmic}
    \State \(\mu \leftarrow 2^{\ceil{\log C}}\), \(\tilde{x} \leftarrow x^0\)
    \Repeat
    \State \(c^\mu \leftarrow \floor{c / \mu}\)
    \State \textbf{compute} \(x \in P\) integral with \(c^\mu(x - \tilde{x}) > 0\)
    \Comment{improve w.r.t.\ \(c^\mu\) approximation of \(c\)} 
    \If {there is no feasible solution}
    \State \(\mu \leftarrow \mu/2\)
    \Else
    \State \(\tilde{x} \leftarrow x\) \Comment{update solution and repeat}
    \EndIf
    \Until{\(\mu < 1\)}
    \State \Return \(\tilde{x}\) \Comment{return optimal solution}
  \end{algorithmic}
\end{algorithm}

Scaling algorithms typically operate in phases: the
algorithm improves the current objective function within a phase as long as
possible and then goes to the next phase: we call \emph{augmentation}
the step where we compute \(x\in P\) that
improves the current objective function and \emph{(scaling) phase} all steps
that use the same scaling factor \(\mu\); bounds are
typically given as a product of an upper bound on the number of phases and an
upper bound on the number of augmentations per phase. 

\citet[Theorem 2]{schulz19950} have proven that Algorithm~\ref{alg:bitScaling}
requires \(O(n \log C)\) augmentation steps with \(C \define \mnorm{c}
+ 1\). We now prove that the number of
augmentation steps is bounded by \(n \cdot (1+\ceil{\log C})\), and also
characterize the absolute gap closed at each scaling phase. In Section
\ref{sec:worst-case-examples} we prove that the bound on the number of 
augmentation steps is tight.

\begin{lemma}[Bit scaling]
  \label{lem:bitScaling}
  Let \(P \subseteq \cube{n}\) be a polytope, and let
  \(c \in \Z^n_+\) with \(C \define \mnorm{c} + 1\) the largest absolute value of
  its components.
  Then Algorithm~\ref{alg:bitScaling} solves the optimization problem
  \(\max \set{cx}{x \in P \cap \Z^n}\) with at most \(n \cdot (1+\ceil{\log C})\)
  augmenting steps.
  Moreover, let \(\tilde{x}\) be the solution at the end of
  scaling phase \(\ell\) and \(x^\star\) be an optimal solution for the
  original problem. Then the absolute gap \(c(x^\star - \tilde{x})\) is bounded by
  \[
  2^{\ceil{\log C} - \ell} \cdot \max \set{\ones x}{x \in P}.
  \]
\end{lemma}

\begin{proof}
  The algorithm applies at most \(1 + \ceil{\log C}\)
  scaling phases, i.e., \(\mu\) is halved at most \(1 + \ceil{\log C}\)
  times. We will show that within each phase, we compute at
  most \(n\) augmenting directions.

  Since \(c^\mu\) is integral, within each phase, we
  improve the previous solution by at least one with respect to \(c^\mu\).  Thus,
  it suffices to compare the initial solution of a given phase with the
  optimal solution of this phase. For this observe that for \(\mu =
  2^{\ceil{\log C}}\) we have \(\mnorm{c^\mu} \leq 1\), and hence the
  absolute gap between an optimal solution \(x^\mu\) for \(c^\mu\) and
  \(x^0\) is at most \(c^\mu(x^\mu - x^0) \leq n\), as \(P \subseteq
  \cube{n}\).

  Next, we will show that if it holds that we
  need at most \(n\) augmenting directions in phase \(\mu\), then we need
  at most \(n\) augmenting directions in phase \(\mu/2\).
  Observe that the objective for phase~\(\mu/2\) is \(c' \define
  \floor{c / (\mu/2)} = \floor{2c / \mu}\), which satisfies 
  \(c' = 2 c^\mu + \tilde c\) for some \(\tilde{c} \in \binSet^n\).
  If now \(x^\mu\) (resp.~\(x'\)) is an optimal solution with respect to \(c^\mu\)
  (resp.~\(c'\)) , we obtain
  \[
  c' (x' - x^\mu) = \underbrace{2 c^\mu (x' - x^\mu)}_{\leq 0} +
  \underbrace{\tilde{c} (x' - x^\mu)}_{\leq n} \leq n.
  \]

  It remains to establish the bound on the gap. Let \(x^\star\) be
  an optimal solution for the original objective function \(c\), and let
  \(x^\mu\) be the solution at the end of phase \(\mu = 2^{\ceil{\log C}
  - \ell}\), i.e., \(x^\mu\) is optimal for \(\floor{c / \mu}\).
  We can write \(c = \floor{c /\mu} \mu + r\) with \(r\in \{0,\dots,
  \mu-1\}^n\) and obtain
  \[
  c(x^\star - x^\mu) = (\floor{c /\mu} \mu + r) (x^\star - x^\mu) = \underbrace{\floor{c /\mu}
    \mu (x^\star - x^\mu)}_{\leq 0} + \underbrace {r (x^\star - x^\mu)}_{\leq
    \mu \max \set{\ones x}{x \in P}} \leq \mu \max \set{\ones x}{x \in P},
  \]
  so the result follows. 
\end{proof}

In particular, when each augmentation phase can be performed fast, then
Algorithm \ref{alg:bitScaling} is fast. We would like to conclude this
section with a few remarks:

\begin{rem}[Efficacy of bit scaling for objective functions with two values]
Note that the idea of bit scaling is
somewhat lost if \(c \in \{0, \gamma\}^n\) for some constant~\(\gamma\),
since in this case only two phases are performed. In particular, the power
of applying bit scaling can be reduced if the objective function~\(cx\) is
incorporated as a constraint \(z_0 = cx\) and \(z_0\) is maximized. Thus,
bit scaling depends heavily on the formulation of the problem. 
\end{rem}

\begin{rem}[General cost functions]
We confined our discussion here to \(c \in \Z^n\)
(and \(c \geq 0\)),
however it can be generalized to arbitrary \(c \in \Q^n\)
using a rounding scheme (via simultaneous Diophantine approximations)
for \(c\)
by \cite{frank1987application}. The same rounding scheme can be used
to ensure that a number of augmentations polynomial in the dimension
is always sufficient (see discussion after
Corollary~\ref{cor:bitscale:worstcase}).  

In the following, we restrict our discussion to \(c
\in \Z^n\) without loss of generality; our implementation works for 
arbitrary cost functions, for details, see Section~\ref{sec:implementation}.
\end{rem}

\begin{rem}[Bit scaling might revisit solutions]
  Bit scaling does not prevent feasible points from being
  revisited. This is because we change the objective function in a way
  that a non-optimal solution for a previous phase could become
  optimal for a later phase, i.e., the sequence of objective functions
  does not induce a unique ordering of the integral points. This
  undesirable behavior will be avoided by the method in the next
  section, and it is this revisiting (or cycling) phenomenon that is at the 
  core of our worst-case example in Section~\ref{sec:worst-case-examples}.
\end{rem}

\subsection{Geometric scaling}
\label{sec:geometric-scaling}

While the analysis of the bit scaling algorithm in Section~\ref{sec:bit-scaling} is geared
towards 0/1 polytopes, we will now present a more general scaling
framework that can be used for integer programs. The generalization to the mixed
integer case will be discussed in Section~\ref{sec:implementation}. The
algorithms in this section are essentially identical to those in
\cite{schulz2002complexity}, with minor modifications to use them in a
framework where the augmentation steps are computed by means of a
mixed-integer program. We will, however, provide a slightly improved analysis of
the geometric scaling algorithm, shaving off a \(\log n\)
factor in comparison to the analysis in
\cite{schulz2002complexity}. This, in particular, establishes that for 0/1 polytopes bit scaling and geometric
scaling have the same worst-case running time in terms of 
augmentation steps. However, as we will see in Section~\ref{sec:worst-case-examples},
there exist instances where geometric scaling requires significantly
fewer augmentations than bit scaling (see Corollary~\ref{cor:bitscale:worstcase}). 

Recall that we aim to solve \(\max \set{cx}{x \in P \cap \Z^n}\)
for an objective function \(c \in \Z^n\)
and a polytope
\(P \define \set{x \in \R^n}{Ax = b,\, l \leq x \leq u}\)
by a sequence of \emph{augmentation steps}; the requirement of
boundedness is important as the boundary will be used for a potential
function. In Section~\ref{sec:geom-scal-arbitr} we use a different
potential function, that does not require equality representations
whenever \(P \subseteq \cube{n}\).

Let us consider a feasible solution \(x \in P \cap \Z^n\). We will compute
an augmenting direction \(z \in \Z^n\) with \(x + z \in P\) and \(cz
> 0\). In the following, we will only consider feasible directions~\(z\), 
i.e., those with \(x + z \in P\), and we will simply call them
\emph{directions}. The (feasible) direction~\(z\) is \emph{exhaustive} for \(x\) if
\(x+2z \notin P\). Note that an exhaustive direction is always nonzero,
and by integrality, an integer feasible direction is exhaustive for \(P\) if
and only if it is exhaustive for the integral hull \(P_I\).

The following scaling algorithm can be understood as
an analogue of interior point methods for integer programs: the objective
function is augmented by a potential function, and the search for
augmenting directions is very similar to the Newton directions obtained
from the derivatives of the classical barrier function for linear programs
(see, e.g., \cite[Section~4]{bental2001lectures}).

\begin{defn}[Potential function]\label{def:potential}
  Let \(P \define \set{x \in \R^n}{Ax = b,\, l \leq x \leq u}\) be a
  polytope. Then \(\rho: \R^n \times \R^n
  \rightarrow \R_+ \cup \{\infty\}\) is a \emph{potential function for \(P\)} if for all
  integer feasible points \(x \in P \cap \Z^n\) and feasible directions
  \(z\) for \(x\):
  \begin{enumerate}
  \item\label{item:1}\(\rho(x, z) \in \order{\poly(n)}\),
  \item\label{item:2} \(\rho(x,z) = \Omega(1/\poly(n))\) whenever \(z\) is
    exhaustive for \(x\), and
 \item \label{item:3}\(\rho(x,\alpha \cdot z) = \alpha\cdot \rho(x, z)\) for all \(\alpha \geq 0\).
  \end{enumerate}
\end{defn}

There are various appropriate potential functions. We now present one of
the potential functions used in \cite{schulz2002complexity}. For this, we
use the standard notation to denote the positive and negative part of \(z\)
by \(z^+\in \Z^n_+\) and \(z^-\in \Z^n_+\), respectively, so that \(z = z^+ - z^-\) and
\(z^+z^- = 0\). We also use the standard convention \(0 \cdot \infty = 0\)
throughout the article.

\begin{lemma}[\cite{schulz2002complexity}]\label{lem:standardPotential}
  Let \(P = \set{x \in \R^n}{Ax = b,\, l \leq x \leq u}\) be a polytope. Then 
  \[
  \rho(x,z) \define p(x) z^+ + n(x) z^-,
  \]
  with 
  \[
  p(x)_j =
  \begin{cases*}
    \frac{1}{u_j - x_j}, & \text{if } \(x_j < u_j\)\\
    \infty, & \text{otherwise}
  \end{cases*}
  \qquad \text{and} \qquad
  n(x)_j =
  \begin{cases*}
    \frac{1}{x_j - l_j}, & \text{if } \(x_j > l_j\)\\
    \infty, & \text{otherwise}\\
  \end{cases*}
  \]
  is a potential function such that
  \begin{enumerate}
  \item \(\rho(x, z) \leq n\) for all integer feasible points
    \(x\) and feasible directions \(z\);
  \item \(\rho(x,z) > \tfrac{1}{2}\) whenever \(z\) is
    exhaustive for \(x\).
  \end{enumerate}
\end{lemma}

\begin{proof}
  Let \(z = z^+ - z^-\) be an integer feasible direction and let \(x\) be
  integer feasible for \(P\). We will show that for each \(j \in
  [n]\) we have \(p(x)_j\, z^+_j + n(x)_j\, z^-_j \leq 1\). Since
  \((p(x)_j\, z^+_j) \cdot (n(x)_j\, z^-_j) = 0\) by the definition
  of the positive and negative part, it suffices to prove that
  \(p(x)_j\, z^+_j \leq 1\) and \(n(x)_j\, z^-_j \leq 1\). We consider
  the term \(n(x)_j\, z^-_j\); the proof is analogous for \(p(x)_j\,
  z^+_j\).
  Observe that whenever \(x_j = l_j\) then \(z^-_j = 0\), and
  hence \(n(x)_j\, z^-_j = 0\) in this case. Thus, suppose that \(x_j >
  l_j\). Then
  \[
  n(x)_j\, z^-_j = \frac{z^-_j}{x_j - l_j} \leq 1,
  \]
  because \(z\) is a (feasible) direction. 

  Now suppose that \(z\) is exhaustive for \(x\), i.e., \(x+z \in
  P\),
  but \(x + 2z \notin P\). By definition, \(\rho(x,z) \geq
  0\). Moreover, since \(z\) is exhaustive, there exists \(j \in [n]\)
  with either \(x_j + 2z_j > u_j\), i.e., \(z^+_j > (u_j - x_j)/2\) or
  \(x_j + 2z_j < l_j\), i.e., \(z^-_j > (x_j-l_j)/2 \). Hence,
  \(p(x)_j\, z^+_j > \tfrac{1}{2}\) in the former case or \(n(x)_j\, z^-_j > \tfrac{1}{2}\) in the
  latter case. 

  Clearly, \(\rho\) as defined above is positively homogeneous in the direction~\(z\),
  i.e., Property~\ref{item:3} is satisfied. 
\end{proof}

Next, we will show that if we can compute a direction that
maximizes the ratio of the objective function value over the potential, we
can solve the maximization problem \(\max \set{cx}{x \in P \cap Z^n}\) by a
number of augmentations polynomial in \(n\) and \(\log C\).  This is achieved by the
maximum-ratio augmentation (MRA) algorithm given in
Algorithm~\ref{alg:mra}, see \cite[Algorithm~I]{schulz2002complexity}. Throughout, let \(C
\define \mnorm{c}\), \(U\define \max_{i\in[n]} u_i\), and \(L \define \min_{i\in[n]} l_i\).

\begin{algorithm}[tb]
  \caption{\label{alg:mra}Maximum-Ratio Augmentation (MRA)}  
  \textbf{Input:} Integer feasible solution \(x^0\), potential function \(\rho\)\\
  \textbf{Output: } Optimal solution for \(\max \set{cx}{x \in P \cap \Z^n}\)
  \begin{algorithmic}
    \State \(\tilde{x} \leftarrow x^0\)
    \Repeat
    \State \textbf{compute} \(x \in \argmax 
    \set{\frac{c(x - \tilde{x})}{\rho(\tilde{x}, x - \tilde{x})}}{c (x - \tilde{x}) > 0,\; x \in P \cap \Z^n}\)
    \Comment{MRA direction}
    \If{there is no feasible solution}
    \State \Return \(\tilde{x}\) \Comment{solution is optimal}
    \Else
    \State \textbf{pick} \(\alpha \in \Z_+\) with \(\alpha \geq 1\) so that \(z = \alpha(x - \tilde{x})\)
    is an exhaustive direction
    \State \(\tilde{x} \leftarrow x + \alpha(x - \tilde{x})\) \Comment{update solution and repeat}
    \EndIf
    \Until{\(\tilde{x}\) is optimal}
  \end{algorithmic}
\end{algorithm}

Observe that we can obtain an exhaustive direction in
Algorithm~\ref{alg:mra} from a maximum-ratio direction \(z = x - \tilde{x}\) simply by scaling up. The scaled direction
will remain an optimal solution to
\begin{equation}\label{eq:MRAobj}
\max \set{\frac{c(x - \tilde{x})}{\rho(\tilde{x}, x - \tilde{x})}}{c (\tilde{x} - x) > 0,\; x \in P \cap \Z^n}
\end{equation}
by Property~\ref{item:3} of Definition~\ref{item:3}, since
\(\frac{c (\alpha z) }{\rho(x,\alpha z)}=\frac{c z }{\rho(x,z)}\).
For a discussion of how to solve~\eqref{eq:MRAobj}, see Section~\ref{sec:Implementation:MRA}.

\begin{rem}
  While a feature of the potential function in
  Lemma~\ref{lem:standardPotential} is the
  guarantee that we do not leave the feasible region, we will not
  require this from a potential function in general (see
  Definition~\ref{def:potential}). In fact, feasibility will be
  ensured by~\eqref{eq:MRAobj}, so that we potentially could
  use a more general class of potential functions, gaining some extra flexibility. 
\end{rem}

\begin{rem}[Equality vs.\ inequality representation]
  It is important to observe that the potential function in
  Lemma~\ref{lem:standardPotential} is defined with respect to the
  equality representation of \(P\).
  In particular, we assign a potential to possible \emph{slack
    variables} of inequalities, which ensures that we move
  \lq{}inside\rq{} the feasible region. This is the reason why we
  require \(P\) to be bounded via \(l \leq x \leq u \); see the discussion
  in \cite[Section 4]{schulz2002complexity}.

In the case of polytopes \(P \subseteq \cube{n}\) we will relax this
in Section~\ref{sec:geom-scal-arbitr} via a different potential
function and provide essentially identical performance guarantees
without adding slack variables. 
\end{rem}

In the remainder of this section, we
will prove the guarantees for the potential function given in
Lemma~\ref{lem:standardPotential}. However, we obtain polynomiality for
any potential function. The results readily carry over by plugging in
the performance values of Properties~\ref{item:1} and \ref{item:2} of
the considered potential function. We will first provide the classical
analysis from \cite{schulz2002complexity} 
for the MRA algorithm as it contains the main potential
function argument that is used throughout the remainder of the article.

\begin{theorem}[Optimization through maximum-ratio augmentation]\label{thm:mraAlgo}
  Consider the polytope \(P = \set{x \in \R^n}{Ax = b,\, l \leq x \leq
    u}\), let \(\rho\) be the potential function from
  Lemma~\ref{lem:standardPotential}, and let \(x^0 \in P \cap \Z^n\).
  Then Algorithm~\ref{alg:mra} solves the optimization
  problem \(\max \set{cx}{x \in P \cap \Z^n}\)
  with at most \(O (n \log( n C (U-L)))\) computations of an MRA direction.
\end{theorem}

\begin{proof}
   Let \(\tilde{x}\) be the solution given at the beginning of an iteration. 
   Suppose an optimal solution \(x\) for \eqref{eq:MRAobj} is found, and
   let \(z = x - \tilde{x}\) be the exhaustive
   direction. Let \(z^\star \define x^\star - \tilde{x}\) be the
   direction from~\(\tilde{x}\) to the optimal solution~\(x^\star\) of the original problem. We have
   \(\frac{c z}{\rho(\tilde{x}, z)} \geq \frac{c z^\star}{\rho(\tilde{x}, z^\star)}\)
   by the optimality of \(x\). It follows that
   \[
   c z \geq \frac{{\rho(\tilde{x}, z)}}{\rho(\tilde{x}, z^\star)} c z^\star \geq
   \frac{1}{2n} c z^\star,
   \]
   since \(\rho(\tilde{x}, z^\star) \leq n\) and \(\rho(\tilde{x}, z) \geq \frac{1}{2}\) by
   Lemma~\ref{lem:standardPotential}. Thus, the computed direction
   recovers a \(\frac{1}{2n}\) fraction of the objective value of the optimal direction
   \(z^\star\). Moreover, the (absolute) gap \(c(x^\star - x^0)\) between an
   optimal solution \(x^\star\) and the initial solution \(x^0\) is at most
   \(K \define n\, C\, (U-L)\).
   Thus, after \(\ell\) rounds, the remaining gap is at
   most \((1-\frac{1}{2n})^\ell K\), and we want to estimate the number of iterations for which
   \((1-\frac{1}{2n})^\ell K \geq 1\) holds; once we drop below \(1\),
   we have reached an integer optimal solution and we are done. Taking the logarithm, we obtain
   \begin{align*}
     \ell \log (1 - \tfrac{1}{2n}) + \log K \geq 0,
   \end{align*}
   which can be bounded using \(\log(1-\frac{1}{2n}) \leq -
   \frac{1}{2n}\). We obtain \(\ell = O (n \log K)\), and the result
   follows.
 \end{proof}

Unfortunately, Algorithm~\ref{alg:mra} computes an exact MRA
direction at each iteration,
which can be expensive and requires maximizing a ratio. We will now consider a scaling algorithm which
approximates the maximum-ratio augmentation direction by a factor of
\(2\) and hence has the same asymptotic running time. 

The main idea of the \emph{geometric scaling} algorithm (see
Algorithm~\ref{alg:mraScaling}) is to only approximately
compute an MRA direction, see \cite[Algorithm~II]{schulz2002complexity}. For this observe that testing whether \(\frac{c 
  z }{\rho(\tilde{x},z)} \geq \mu\) is equivalent to testing \(cz - \mu \cdot
\rho(\tilde{x},z) \geq 0\). This is precisely the standard methodology of the
barrier method to progressively tighten the complementary slackness
residual. 

\begin{algorithm}[tb]
  \caption{Geometric scaling}
  \label{alg:mraScaling}
  \textbf{Input: } Integer feasible solution \(x^0\), potential function \(\rho\)\\
  \textbf{Output: } Optimal solution for \(\max \set{cx}{x \in P \cap \Z^n}\)
  \begin{algorithmic}
    \State \(\mu \leftarrow 2C(U-L)\), \(\tilde{x} \leftarrow x^0\)
    \Repeat
    \State \textbf{compute} \(x \in P\) integral with \(c(x - \tilde{x}) - \mu \cdot
    \rho(\tilde{x}, x - \tilde{x}) > 0\) \Comment{approx. MRA direction} 
    \If{there is no feasible solution}
    \State \(\mu \leftarrow \mu/2\)
    \Else
    \State \textbf{pick} \(\alpha \in \Z_+\) with \(\alpha \geq 1\) so that \(z = \alpha(x - \tilde{x})\)
    is an exhaustive direction
    \State \(\tilde{x} \leftarrow \tilde{x} + \alpha(x - \tilde{x})\) \Comment{update solution and repeat}
    \EndIf
    \Until{\(\mu < 1/n\)}
    \State \Return \(\tilde{x}\)\Comment{return optimal solution}
  \end{algorithmic}
\end{algorithm}

Again, we can scale the direction to be exhaustive due to the homogeneity of the potential function. 
In order to establish a performance bound for Algorithm~\ref{alg:mraScaling} we need the following
simple observation.

\begin{observation}\label{obs:decreasingPot}
  Let \(\tilde{x}\) be the last solution in the scaling phase for~\(\mu\). Then
  for any integer feasible solution~\(x\), \(\tilde{x}\) satisfies
  \[
  \frac{c(x - \tilde{x})}{\rho(\tilde{x}, x - \tilde{x})} \leq \mu,
  \]
  i.e., whenever we enter a new scaling phase we have \(c(x - \tilde{x})
  \leq \mu\, n\) for the potential function in
  Lemma~\ref{lem:standardPotential}, which gives an upper bound on the
  remaining gap.
\end{observation}

Moreover let us point out the following property:

\begin{observation}[Geometric scaling never revisits a point]\label{obs:GSnoRevisit}
  A feasible solution~\(x\) in Algorithm~\ref{alg:mraScaling} satisfies
  \(c(x - \tilde{x}) - \mu \cdot \rho(\tilde{x}, x - \tilde{x}) > 0\) or equivalently,
  \[
  cx > c\tilde{x} + \mu \cdot \underbrace{\rho(\tilde{x},x-\tilde{x})}_{\geq 0}.
  \]
  Thus, geometric scaling produces solutions with strictly increasing
  cost with respect to the original objective function~\(c\) and cannot revisit points.
\end{observation}

The crucial advantage of the geometric scaling algorithm is that it
uses the objective function to guide the search and hence we (potentially) obtain a speed-up over standard augmentation. 
For illustration purposes, we depict the behavior of the geometric
scaling algorithm in Figure~\ref{fig:geoFig}.

\begin{figure}
  \centering
  \IfFileExists{graphics/gsExample_n10.pdf}{
     \includegraphics[scale=0.8]{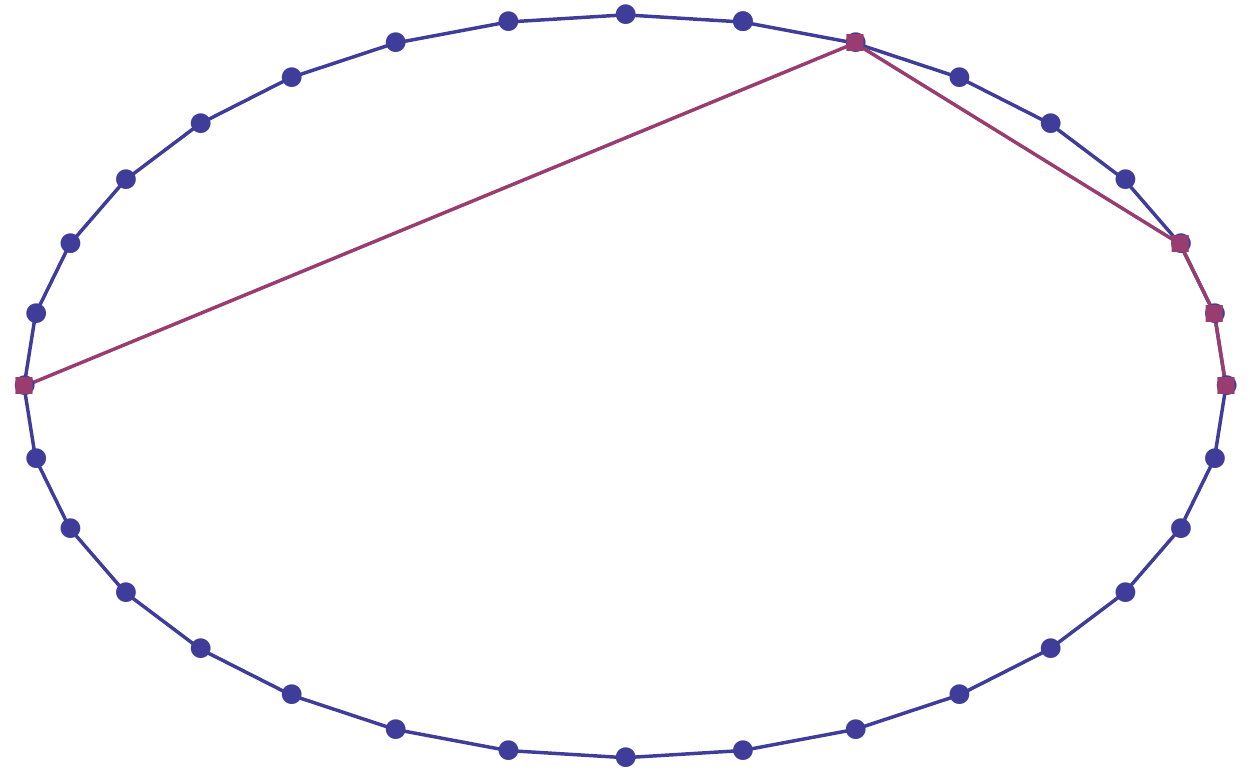}
  }
  {
  \IfFileExists{gsExample_n10.pdf}{
     \includegraphics[scale=0.8]{gsExample_n10.pdf}
  }{}
  }
  \caption{Illustration of the behavior of geometric scaling. The
    picture shows a polytope in \(\R^2\).
    Our initial point is the leftmost point and we maximize
    \(c=(1,0)\)
    (i.e., the maximum is the rightmost point). A worst-case
    augmentation oracle might give the point with smallest
    improvement, which is always an adjacent vertex or jump between
    bottom and top. Standard
    augmentation with such a malicious worst-case oracle would now
    force us to visit either each point on the upper or lower
    path. Geometric scaling constructs shortcuts by controlling the
    objective function, leading to significant speed-ups.}
  \label{fig:geoFig}
\end{figure}

As with bit scaling and MRA, Algorithm~\ref{alg:mraScaling} also
requires only polynomially many augmentations with respect to
the encoding length. First we bound the number of augmentations required
per scaling phase.

\begin{lemma}[\cite{schulz2002complexity}]\label{lem:updatesPerPhase}
  Let \(P = \set{x \in \R^n}{Ax = b,\, l \leq x \leq u}\),
  let \(\rho\) be the potential function from
  Lemma~\ref{lem:standardPotential}, and let \(x^0 \in P \cap \Z^n\) be an integer feasible
  solution. Then Algorithm~\ref{alg:mraScaling} computes at most \(4n\) approximate MRA directions between successive updates of \(\mu\).
\end{lemma}
\begin{proof}
  Let \(y^0, y^1,  \dotsc\) be the points in \(P\) visited by the algorithm during the scaling phase for a given~\(\mu\). In particular,
  \(y^0\) is the current solution after the last update of \(\mu\). By Observation~\ref{obs:decreasingPot}, we have
  \[
  \frac{c(x^\star - y^0)}{\rho(y^0,x^\star - y^0)} \leq 2\mu,
  \]
  where \(x^\star\) is an integral optimal solution for the original problem. Now, consider any two
  consecutive iterates \(y^i\) and \(y^{i+1}\). By definition of
  Algorithm~\ref{alg:mraScaling}, the \(c(y^{i+1} - y^i) - \mu \cdot
  \rho(y^i, y^{i+1} - y^i) > 0\) holds. Moreover, as the direction
  \(y^{i+1} - y^i\) is exhaustive, using Lemma~\ref{lem:standardPotential}, we have 
  \begin{align*}
    c(y^{i+1}-y^i) & > \mu \cdot \rho(y^{i},y^{i+1}-y^i)  \geq \frac{\mu}{2}
    \geq \frac{1}{4} \frac{c(x^\star-y^0)}{\rho(y^0,x^\star-y^0)} \geq
    \frac{1}{4n} c(x^\star-y^0),
  \end{align*}
  hence we compute at most \(4n\) approximate directions in each scaling phase.
\end{proof}

It is interesting to observe that in
Lemma~\ref{lem:updatesPerPhase}, in each scaling phase, we recover at
least a~\(\tfrac{1}{4n}\) fraction of the improvement of the optimal direction \(x^\star-y^0\)
from the feasible solution \(y^0\) at the beginning of the phase to the
optimal solution \(x^\star\). This is in contrast to Theorem~\ref{thm:mraAlgo}, where the guaranteed
improvement of \(\tfrac{1}{2n}\) is with respect to two consecutive iterates 
only, i.e., we only guarantee to recover an \(\tfrac{1}{2n}\) fraction of
the improvement of \(x^\star-y^{i-1}\), if we are in iteration \(i\), which is potentially
smaller than the one from direction  \(x^\star-y^0\).

We will now establish a bound on the number of required approximate MRA
directions, which slightly improves the bound in
\cite{schulz2002complexity} by a \(\log n\) factor. The key insight is
that we can combine 
Observation~\ref{obs:GSnoRevisit} with
Observation~\ref{obs:decreasingPot}, to switch from the multiplicative
regime to the additive regime, simply counting the remaining
improvement steps.  

\begin{theorem}[Improved bound for geometric scaling]\label{thm:mraScalingRt}
 Let \(P = \set{x \in \R^n}{Ax = b,\, l \leq x \leq u}\),
 let \(\rho\) be the potential function from
 Lemma~\ref{lem:standardPotential}, and let \(x^0 \in P \cap \Z^n\) be an integer feasible
 solution. Then Algorithm~\ref{alg:mraScaling} solves the optimization
 problem \(\max \set{cx}{x \in P \cap \Z^n}\)
 with at most \(\order{n \log(C (U-L))}\) computations of approximate MRA
 directions.
\end{theorem}

\begin{proof}
  The algorithm initializes with \(\mu = 2 C(U-L)\). Hence after
  \(\ceil{\log(C(U-L))}+1\) updates of~\(\mu\),
  we have \(\mu\leq1\), and by
  Lemma~\ref{lem:updatesPerPhase} have computed at most
  \(4n(\ceil{\log(C(U-L))}+1)\) approximate MRA directions in total. Let \(\tilde{x}\) be
  the last solution computed by the algorithm in these first
  \(\ceil{\log(C(U-L))}+1\) scaling phases.

  We now switch to the additive regime and simply count the number of
  remaining improvements that are possible. As \(\mu\leq1\), Observation~\ref{obs:decreasingPot} implies
  \[
  c(x^\star - \tilde{x}) \leq \mu \cdot \rho(\tilde{x}, x^\star -
  \tilde{x}) \leq n,
  \]
  where \(x^\star\) is an integral optimal solution with respect to \(c\). Since all data is
  integral and by Observation~\ref{obs:GSnoRevisit}, every approximate
  MRA direction leads to an improvement of the objective function by at
  least \(1\). It follows that no more than \(n\) solutions may be
  generated before obtaining a solution with cost \(cx^\star\). Hence the
  algorithm terminates after computing at most
  \(4n(\ceil{\log(C(U-L))}+1)+n\) approximate MRA directions.
\end{proof}

\begin{rem}[Oracle calls vs.~approximate MRA directions] In
  Theorem~\ref{thm:mraScalingRt} and elsewhere we count the number of
  approximate MRA directions that we compute. That is slightly
  different than counting the number of calls to an approximate
  MRA oracle: we do not count the number of calls for which no
  approximate MRA direction exist for a given scaling factor \(\mu\)
  and where \(\mu\)
  is rescaled. However, note that this number of calls is dominated by the
  number of calls which do return improving directions. For example in
  Theorem~\ref{thm:mraScalingRt}, in the last phase where
  \(\mu \leq 1\), we rescale at most \(O(\log n) = o(n)\) times until \(\mu
  < 1/n\). In fact, all our results also hold (up to constant factors) if we
  consider the number of oracle calls rather than approximate MRA directions.
\end{rem}

Note that the bound in Theorem~\ref{thm:mraScalingRt} is stronger than the one given in
Theorem~\ref{thm:mraAlgo} for Algorithm~\ref{alg:mra}. In fact, the
above result implies the same worst-case bound for
Algorithm~\ref{alg:mra}: As Algorithm~\ref{alg:mraScaling} may use a
ratio-maximizing direction in each step,
Algorithm~\ref{alg:mra} inherits any worst-case upper bounds proven for
Algorithm~\ref{alg:mraScaling}. Thus we obtain the following improvement:

\begin{corollary}[Improved bound for MRA]\label{cor:mraAlgoTightened}
  Let \(P = \set{x \in \R^n}{Ax = b,\, l \leq x \leq u}\), \(\rho\) be the potential function from
  Lemma~\ref{lem:standardPotential}, and let \(x^0 \in P \cap \Z^n\).
  Then Algorithm~\ref{alg:mra} solves the optimization
  problem \(\max \set{cx}{x \in P \cap \Z^n}\)
  with at most \(\order{n \log(C (U-L))}\) computations of an MRA direction.
\end{corollary}

Moreover, in the case of 0/1 polytopes where the description of the LP
relaxation is in equality form, we obtain: 

\begin{corollary}[Worst-case performance for 0/1 polytopes]
\label{cor:worstCaseIdentical}
  Let \(P = \set{x \in \R^n}{Ax = b,\, 0 \leq x \leq \ones}\), \(\rho\) be the potential function from
  Lemma~\ref{lem:standardPotential}, and let \(x^0 \in P \cap \binSet^n\).
  Then Algorithms~\ref{alg:mra} and  \ref{alg:mraScaling} both solve the optimization
  problem \(\max \set{cx}{x \in P \cap \binSet^n}\)
  with at most \(\order{n \log C}\) augmentations.
\end{corollary}

\subsection{Geometric scaling for arbitrary polytopes in the 0/1 cube}
\label{sec:geom-scal-arbitr}

We will now briefly explain how the setup from above can be changed in the
case of polytopes \(P = \face{x \mid Ax \leq b} \subseteq \cube{n}\), i.e.,
not requiring equality form. Clearly, \(P\) can be written in equality form
by adding slack variables. This, however, changes the ambient dimension,
which affects all the bounds above. Moreover, slack variables do not
necessarily have to be 0/1 variables, complicating things further. Thus, we
present a tailored analysis for problems in the above form with a
particular potential function.

The following observation is crucial:

\begin{observation}[Exhaustiveness for 0/1 polytopes]\label{obs:01exhaust}
Every 0/1 solution \(x \in
P\) is a vertex of \(P\), and, in particular, for each coordinate either
\(0 \leq x\) or \(x \geq \ones\) is tight, so that any direction \(0\neq z\in
\face{-1,0,1}^n\) is always exhaustive. For the potential function \(\rho(x,z)
\define \onorm{z}\), clearly \(\rho(x,z) \leq n\), and
\(\rho(x,z) \geq 1\) whenever \(z \neq 0\). Moreover, we do not
need homogeneity, as no scaling of directions is required.
\end{observation}
 
We obtain the following lemma with a proof essentially identical to the
one for Lemma~\ref{lem:updatesPerPhase}.

\begin{lemma}\label{lem:updatesPerPhaseAlt}
  Let \(P = \set{x \in \R^n}{Ax \leq b,\, 0 \leq x \leq \ones}\), let
  \(\rho(x,z) = \onorm{z}\), and let \(x^0 \in P \cap \face{0,1}^n\) be a feasible 0/1
  solution. Then Algorithm~\ref{alg:mraScaling} computes at most \(2n\)
  approximate MRA directions between successive updates of \(\mu\).
\end{lemma}

\begin{proof}
  The beginning of the proof is as in Lemma~\ref{lem:updatesPerPhase}, but now two consecutive iterates \(y^i,y^{i+1}\) within a scaling phase satisfy
  \begin{align*}
    c(y^{i+1}-y^i) &\geq \mu\cdot \rho(y^{i},y^{i+1}-y^i)  \geq \mu
    \geq \frac{1}{2} \frac{c(x^\star-y^0)}{\rho(y^0,x^\star-y^0)} \geq
    \frac{1}{2n} c(x^\star-y^0),
  \end{align*}
  by Observation~\ref{obs:01exhaust}, where \(x^\star\)
  is an integral optimal solution with respect to \(c\).
  Hence at most \(2n\)
  approximate MRA directions are computed in each scaling phase.
\end{proof}

With this lemma we obtain the following version of
Theorem~\ref{thm:mraScalingRt} for \emph{arbitrary} polytopes \(P
\subseteq \cube{n}\). The proof follows exactly as in
Theorem~\ref{thm:mraScalingRt}, but with Lemma~\ref{lem:updatesPerPhaseAlt}
playing the role of Lemma~\ref{lem:updatesPerPhase}.

\begin{theorem}
\label{thm:geoScalGen01}
 Let \(P = \set{x \in \R^n}{Ax \leq b,\, 0 \leq x \leq \ones}\),
 let \(\rho(x,z) = \onorm{z}\), and let \(x^0 \in P \cap \face{0,1}^n\) be a feasible 0/1
 solution. Then Algorithm~\ref{alg:mraScaling} solves the optimization
 problem \(\max \set{cx}{x \in P \cap \binSet^n}\)
 with at most \(\order{n \log C }\) computations of approximate MRA
 directions. 
\end{theorem}

In particular, the computation of the approximate MRA direction can be
performed with a \emph{single call} to an augmentation oracle as the
resulting program with \(\onorm{\cdot}\) as potential function can be
phrased as an integer program. Thus, bit scaling and geometric scaling
require essentially the same number of computations of augmenting
steps (see Lemma~\ref{lem:bitScaling}). The following is a
generalization of Corollary~\ref{cor:worstCaseIdentical} in the case
of \(P \subseteq \cube{n}\).

\begin{corollary} 
\label{cor:identicalWorstCaseGen01} 
  Let \(P = \set{x \in \R^n}{Ax \leq b,\, 0 \leq x \leq \ones} \subseteq
  \cube{n}\) be a polytope and
  consider the potential function \(\rho(x,z) \coloneqq \onorm{z}\).
  Let  \(x^0 \in P \cap \binSet^n\) be an arbitrary integral solution.
  Then Algorithms~\ref{alg:mra} and  \ref{alg:mraScaling} both solve the optimization
  problem \(\max \set{cx}{x \in P \cap \binSet^n}\)
  with at most \(\order{n \log C}\) augmentations.
\end{corollary}

We will now give an intuition for the result above. In fact, it turns
out that the geometric scaling algorithm
(Algorithm~\ref{alg:mraScaling}) and bit scaling
(Algorithm~\ref{alg:bitScaling}) are closely related:

\begin{rem}[Relation between bit scaling and geometric scaling]
  Observe that the potential function from
  Lemma~\ref{lem:standardPotential} is equivalent to the potential
  function
  \(\rho(\tilde{x}, x - \tilde{x}) \define \abs{\supf{x - \tilde{x}}}\)
  in the 0/1 case, provided we consider only feasible directions. Now
  consider a polytope \(P \subseteq \cube{n}\),
  an integral objective function \(c \in \Z^n_+\)
  (which we can assume to be nonnegative by flipping), an integer
  feasible point \(\tilde{x} \in P \cap \face{0,1}^n\),
  and \(\mu = 2^\ell\)
  for some \(\ell \in \N\).
  In Algorithm~\ref{alg:mraScaling} we search for a direction~\(z\)
  defined by \(z = x - \tilde{x}\),
  where \(x \in P\) is integer feasible so that
\[
c(x - \tilde{x}) - \mu\, \abs{\supf{x - \tilde{x}}} > 0.
\]
If we would now pick any coordinate \(j \in [n]\), then the above stipulates that it is only
beneficial to deviate from the \(\tilde{x}_j\) value if \(c_j >
2^\ell\). Writing \(c = c^1 + c^0\) with \(c^1 \define \floor{c/2^\ell} \cdot 2^\ell\) and \(c^0
\define c - c^1\), we obtain
\[
c^1(x - \tilde{x}) + \underbrace{c^0(x - \tilde{x}) - 2^\ell\, \abs{\supf{x - \tilde{x}}}}_{\leq 0 }>0.
\]
Hence, \(c^1(x-x_0) > 0\) is a necessary condition. Although this condition does
not guarantee an improvement over~\(\tfrac{cx}{\rho(\tilde{x}, x - \tilde{x})}\),
in each phase at most \(n\) augmentation steps are necessary
(see the analysis in the proof of Lemma~\ref{lem:bitScaling}), leading
virtually to the same overall running time as for
Algorithm~\ref{alg:mraScaling}, however with the additional
simplification of not explicitly having to consider the potential.   
\end{rem}

\subsection{Improved bounds for structured 0/1 polytopes}

When proving worst-case bounds for both bit scaling and geometric
scaling, a crucial element is the \(\order{n}\) bound on the number of
improvements made per scaling phase. In the case of bit scaling, this bound
is due to the number of positive entries in the vector \(x - \tilde{x}\)
being at most \(n\) for any integral point \(x\), \(\tilde{x} \in
P\). For geometric scaling, the bound arises from potential function
values. In particular, the potential \(\rho(x,z)
\coloneqq \onorm{z}\) is bounded from above by
\(n\). If this bound can be reduced for special polytopes, it would have direct
consequences for worst-case bounds of either algorithm.

One condition that guarantees such a reduction is the
following: Let \(P\subseteq[0,1]^n\) be a polytope, and suppose there exists some function \(f:\Z_+\rightarrow\Z_+\)
such that every integral point \(x\in P\) has no more than \(f(n)\)
nonzero entries. In particular we are hoping for an \(o(n)\) function,
such as \(\sqrt{n}\) or \(\log n\). We then obtain the following improved worst-case
bounds for both bit scaling and geometric scaling.

\begin{theorem}\label{thm.limitedNonzeros}
  Let \(c\in\R^n\) be a cost vector, \(P\subseteq[0,1]^n\) a polytope, and let
  the potential function \(\rho\) be given as \(\rho(x,z) \coloneqq
  \onorm{z}\). Suppose there exists a function \(f: \Z_+ \to \Z_+\) such
  that every integral point \(x\in P\) has at most \(f(n)\) nonzero entries. Then,
  given an initial solution \(x^0\in P\), Algorithms~\ref{alg:bitScaling} and
  \ref{alg:mraScaling} solve the optimization problem \(\max \set{cx}{x \in
    P \cap \Z^n}\) after \(\order{f(n)\log C}\) augmentations.
\end{theorem}

\begin{proof}
  For Algorithm~\ref{alg:bitScaling} the proof follows as for
  Lemma~\ref{lem:bitScaling}: suppose \(x^\mu\in P\) optimizes
  \(c^\mu=\floor{c/\mu}\) over \(P\), and we move to the next scaling phase
  (dividing \(\mu\) by 2) in which we optimize over \(c'=2c^\mu+\tilde c\)
  for some \(\tilde c\in\{0,1\}^n\). If we take \(x'\) to be an optimal
  solution with respect to \(c'\), we have
  \[
  c' (x' - x^\mu) = \underbrace{2 c^\mu (x' - x^\mu)}_{\leq 0} +
  \underbrace{\tilde{c} (x' - x^\mu)}_{\leq 2\, f(n)} \leq 2 f(n),
  \]
  since the direction \(x' - x^\mu \in\{-1,0,1\}^n\) has at most
  \(2\, f(n)\) positive entries. Hence, no more than~\(2 f(n)\)
  improvements can be made in any of the \(\ceil{\log C}+1\) scaling
  phases.

  In the case of Algorithm~\ref{alg:mraScaling}, we know that \(\rho\)
  is bounded above by \(2f(n)\).
  From this point forward, the proof follows exactly as for
  Theorem~\ref{thm:mraScalingRt}, with requiring in total at most
  \(8f(n)(\ceil{\log(C)}+1)+2f(n)\) approximate MRA directions.
\end{proof}

Many well-studied polytopes satisfy the structural constraint from
Theorem~\ref{thm.limitedNonzeros}, especially those arising from
graph-theoretic problems. For example, take the traveling salesman polytope
\(P\subseteq[0,1]^E\)
on the complete graph with~\(k\)
nodes and \(\card{E} = \binom{k}{2}\)
edges. Even though the polytope is contained in a space of ambient
dimension \(\binom{k}{2}\),
its integral points (corresponding to tours on the graph)
contain exactly \(k\)
nonzero entries, spanning a low dimensional subspace. Hence optimizing
over \(P\)
using either Algorithm~\ref{alg:bitScaling} or
Algorithm~\ref{alg:mraScaling} can be done in \(\order{k\log C}\)
augmentations, a factor-\(k\) improvement over the general
\(\order{k^2 \log C}\)
upper bound. Similar statements hold for the case of maximum weight
matchings on a complete graph.

\section{Worst-case example for bit scaling}
\label{sec:worst-case-examples}

We will now show that the upper bound in Lemma~\ref{lem:bitScaling}, on the number of augmentations necessary for bit scaling, is tight. For this we provide a
family of polytopes \(P_n \subseteq \cube{n}\) and cost functions
\(c^p\) so that the bit scaling
method needs \(\Omega(n \log \mnorm{c^p})\) augmentation steps in the worst case. 

Each instance of this family is parametrized by two numbers, namely \(k \in \Z_{+}\), which dictates the dimension \(n \define 8k-2\) of the cube \(\cube{n}\),
and \(p\in\Z_{+}\), which controls how the objective function~\(c^p\) is built, and, by construction, the number \(p\) of bit scaling phases that will be required to solve the instance.

\subsection{Construction of the polytope}
\label{sec:constr-polyt}

The polytope \(P_n\subseteq[0,1]^n\) will be of the form
\[
P_n=\conv{\face{y^1, \dots, y^{2k}}},
\]
where the vectors \(y^j\in\{0,1\}^n\) are defined in terms of vectors
\(y^{j,1}\in\{0,1\}^{k-1},y^{j,2}\in\{0,1\}^{k-1},y^{j,3}\in\{0,1\}^{3k}\), and
\(y^{j,4}\in\{0,1\}^{3k}\). With these four families of vectors defined, the full vector \(y^j\) is given by
\[
y^j \define \begin{pmatrix}
y^{j,1}\\y^{j,2}\\y^{j,3}\\y^{j,4}
\end{pmatrix}
\qquad\text{or equivalently}\qquad
y^j_i \define \begin{cases}
y^{j,1}_i&\text{for }i\in\{1, \dots, k-1\},\\
y^{j,2}_{i-k+1}&\text{for }i\in\{k, \dots, 2k-2\},\\
y^{j,3}_{i-2k+2}&\text{for }i\in\{2k-1, \dots, 5k-2\},\\
y^{j,4}_{i-5k+2}&\text{for }i\in\{5k-1, \dots, 8k-2\}.
\end{cases}
\]
The parts \(y^{j,1},y^{j,2}\) are defined in two batches.  For the first batch with \(j\in
\face{1,\dots,k}\), we define
\[
y^{j,1}_i \define \begin{cases}
1&\text{if }i\geq j,\\
0&\text{otherwise},
\end{cases}
\quad\quad
y^{j,2}_i \define \begin{cases}
1&\text{if }i<j,\\
0&\text{otherwise}
\end{cases}
\quad\text{ for }i = 1, \dots, k-1.
\]
For the second batch with \(j\in \face{k+1,\dots,2k}\), we define
\[
y^{j,1}_i \define \begin{cases}
1&\text{if }i\geq j-k,\\
0&\text{otherwise},
\end{cases}
\quad\quad
y^{j,2}_i \define\begin{cases}
1&\text{if }i<j-k,\\
0&\text{otherwise},
\end{cases}
\quad\text{ for }i = 1, \dots, k-1.
\]
We define \(y^{j,3},y^{j,4}\) with \(j\in \face{1,\dots,2k}\) as follows
\[
y^{j,3}_i \define \begin{cases}
1&\text{if }j\leq k,\\
0&\text{otherwise},
\end{cases}
\quad\quad
y^{j,4}_i \define\begin{cases}
1&\text{if }j>k,\\
0&\text{otherwise},
\end{cases}
\quad\text{ for }i = 1, \dots, 3k.
\]
See Figure~\ref{pointmatrix} for an illustration.

\subsection{Construction of the cost vector}
\label{sec:constr-cost-vect}

The cost vector is defined inductively, keeping the mechanics of the
bit scaling procedure in mind. We first define \(c^0 \define 0\), and for \(\ell = 1, \dots, p\), we build \(c^\ell=2c^{\ell-1}+d^\ell\), for some vector \(d^\ell\in\{0,1\}^n\) to be specified. We will find it convenient to construct \(d^\ell = (d^{\ell,1}, d^{\ell,2}, d^{\ell,3},d^{\ell,4})\) in terms of
vectors \(d^{\ell,1}, d^{\ell,2}, d^{\ell,3}\), and \(d^{\ell,4}\) in the same manner
as we did for the points \(y^j\). 

For \(d^1 \define c^1\), let
\[
d^{1,1} \define \ones,\quad\quad
d^{1,2} \define 0,\quad\quad
d^{1,3}_i \define \begin{cases}
1&\text{if }i\leq k,\\
0&\text{otherwise},
\end{cases}
\text{ for } i = 1, \dots, 3k,
\quad\quad
d^{1,4} \define 0.
\]
For \(\ell \geq 2\), we set
\[
d^{\ell,1} \define 0,\quad\quad d^{\ell,2} \define \ones,\quad\quad
d^{\ell,3} \define \begin{cases}
\ones&\text{if }\ell \text{ is odd},\\
0&\text{otherwise},
\end{cases}
\quad\quad
d^{\ell,4} \define \begin{cases}
\ones&\text{if }\ell \text{ is even},\\
0&\text{otherwise}.
\end{cases}
\]
In particular, after the first scaling phase, the contribution of the
first \(2(k-1)\) coordinates is the same for all
\(y^j\). In fact, we use the first \(2(k-1)\) coordinates for the
improvements steps within a scaling phase and the last \(6k\)
coordinates to switch between the phases; this will become clear
soon. Note that for each \(\ell>1\), \(\log\mnorm{c^\ell}\in\Theta(\ell)\).

\subsection{Lower bound on the number of augmentations}
\label{sec:lower-bound-number}

We will now derive a lower bound on the worst-case number of augmentations computed by the bit scaling algorithm when applied to a polytope \(P_n\) and cost vector \(c^p\) as defined in Section \ref{sec:constr-polyt} and \ref{sec:constr-cost-vect}, respectively. We depict the overall structure of the construction in
Figure~\ref{pointmatrix}, describing the points \(y^j\) and the
``layers'' \(d^\ell\) of the cost function. Note how the columns in Figure~\ref{pointmatrix} are divided into four segments. These four segments correspond to the four families of vectors used in defining~\(y^j\) and \(d^\ell\). For example, the first group of columns in the \(y^j\) row depict the vector \(y^{j,1}\), the second group of columns depict \(y^{j,2}\), and so on.

\begin{figure}[tb]
\[
\begin{array}{c|ccccc|ccccc|cccc|cccc}\toprule
& \multicolumn{5}{c|}{y^{j,1}} & \multicolumn{5}{c|}{y^{j,2}} & \multicolumn{4}{c|}{y^{j,3}}& \multicolumn{4}{c}{y^{j,4}}\\\midrule
&1&2&3&\cdots&k-1&1&2&3&\cdots&k-1&1&2&\cdots&3k&1&2&\cdots&3k\\\midrule
y^1&1&1&1&\cdots&1&0&0&0&\cdots&0&1&1&\cdots&1&0&0&\cdots&0\\
y^2&0&1&1&\cdots&1&1&0&0&\cdots&0&1&1&\cdots&1&0&0&\cdots&0\\
y^3&0&0&1&\cdots&1&1&1&0&\cdots&0&1&1&\cdots&1&0&0&\cdots&0\\
\vdots&\vdots&\vdots&\vdots&\ddots&\vdots&\vdots&\vdots&\vdots&\ddots&\vdots&\vdots&\vdots&\ddots&\vdots&\vdots&\vdots&\ddots&\vdots\\
y^k&0&0&0&\cdots&0&1&1&1&\cdots&1&1&1&\cdots&1&0&0&\cdots&0\\\midrule
y^{k+1}&1&1&1&\cdots&1&0&0&0&\cdots&0&0&0&\cdots&0&1&1&\cdots&1\\
y^{k+2}&0&1&1&\cdots&1&1&0&0&\cdots&0&0&0&\cdots&0&1&1&\cdots&1\\
y^{k+3}&0&0&1&\cdots&1&1&1&0&\cdots&0&0&0&\cdots&0&1&1&\cdots&1\\
\vdots&\vdots&\vdots&\vdots&\ddots&\vdots&\vdots&\vdots&\vdots&\ddots&\vdots&\vdots&\vdots&\ddots&\vdots&\vdots&\vdots&\ddots&\vdots\\
y^{2k}&0&0&0&\cdots&0&1&1&1&\cdots&1&0&0&\cdots&0&1&1&\cdots&1\\\midrule
d^1&1&1&1&\cdots&1&0&0&0&\cdots&0&d^{1,3}_1&d^{1,3}_2&\cdots&d^{1,3}_{3k}&0&0&\cdots&0\\
d^2&0&0&0&\cdots&0&1&1&1&\cdots&1&0&0&\cdots&0&1&1&\cdots&1\\
d^3&0&0&0&\cdots&0&1&1&1&\cdots&1&1&1&\cdots&1&0&0&\cdots&0\\
d^4&0&0&0&\cdots&0&1&1&1&\cdots&1&0&0&\cdots&0&1&1&\cdots&1\\
d^5&0&0&0&\cdots&0&1&1&1&\cdots&1&1&1&\cdots&1&0&0&\cdots&0\\
\vdots&\vdots&\vdots&\vdots&\vdots&\vdots&\vdots&\vdots&\vdots&\vdots&\vdots&\vdots&\vdots&\vdots&\vdots&\vdots&\vdots&\vdots&\vdots\\
\bottomrule
\end{array}
\]\caption{Structure of \(y^j\) and \(d^\ell\); note that \(d^{1,3}\) depends on \(k\).}\label{pointmatrix}
\end{figure}

The essence of the proof is the following: within a scaling phase, the algorithm may move to \emph{any} solution with an improving cost with respect to vector \(\floor{c/\mu}\) (recall that \(\mu\) is the scaling factor), no matter the magnitude of the improvement. In our construction, no matter the choice of \(k,\ell\), the bit scaling algorithm begins by optimizing over the cost vector \(c^1\). The construction is such that \(c^1y^1>c^1y^2> \dots >c^1y^{2k}\). Thus if the algorithm begins at initial solution \(y^{2k}\), it may visit \emph{all} of the~\(2k\) points in \(P_n\), ending the initial phase at~\(y^1\).

In the second scaling phase, the algorithm optimizes over \(c^2\). We will see that we have \(c^2y^1<c^2y^{2k}<c^2y^{2k-1}< \dots <c^2y^{k+1}\). Thus, in this phase, the algorithm may take \(k\) augmentation steps before finishing at point \(y^{k+1}\). In the third augmentation phase, while optimizing over \(c^3\), we similarly have \(c^3y^{k+1}<c^3y^k< \dots <c^3y^1\), giving another possible \(k\) augmentations within the phase.

The process continues in each subsequent scaling phase, with the algorithm having the opportunity to travel through each of the points \(y^{2k},y^{2k-1}, \dots, y^{k+1}\) in even phases, and \(y^k,y^{k-1}, \dots, y^1\) in odd phases, as depicted in Figure~\ref{fig:worstCase}. Since \(k\approx n/8\), this implies a worst case \(\order{n}\) augmentations per scaling phase, meeting the upper bound from Lemma~\ref{lem:bitScaling}.

We now begin the formal proof. We will first show that in each phase \(\ell\), the first~\(k\) points
\(y^1,\dots, y^k\) are ordered
in a decreasing fashion by the objective function \(c^\ell\) and
similar for the second~\(k\) points \(y^{k+1}, \dots, y^{2k}\). In
a second step we will then link the two groups. 

\begin{lemma}[Decreasing order within each group]\label{lem:intragroupOrder}
  Let \(c^\ell,y^j\) be constructed as above. For any \(\ell \geq1\) and \(j
  \in\face{1,\dots,k-1}\cup\face{k+1,\dots,2k-1}\), we have
  \[
  c^\ell y^j=c^\ell y^{j+1}+1.
  \]
\end{lemma}

\begin{proof}
  The proof is by induction on \(\ell\), with base case \(\ell =1\). For \(j\in\{1, \dots, 2k\}\), define \(\alpha_{1,j} \define d^{1,3}
      y^{j,3}\). For \(j\in\{1, \dots, 2k\}\), we have
  \begin{align*}
    c^1 y^j = d^1 y^j = \underbrace{d^{1,1} y^{j,1}}_{= k-j} +
    \underbrace{d^{1,2} y^{j,2}}_{= 0} + \underbrace{d^{1,3}
      y^{j,3}}_{= \alpha_{1,j}} + \underbrace{d^{1,4} y^{j,4}}_{=0}
    = k - j + \alpha_{1,j}.
  \end{align*}
  By construction, we have \(\alpha_{1,1}=\alpha_{1,2}= \dots =\alpha_{1,k} = 1\) and \(\alpha_{1,k+1}=\alpha_{1,k+2}= \dots =\alpha_{1,2k}=0\). Thus for \(j \in \face{1,\dots,k-1}\cup\face{k+1,\dots,2k-1}\),  we can establish
  \[
  c^1 y^j - c^1 y^{j+1} = k-j+\alpha_{1,j} - (k-(j+1)+\alpha_{1,j+1})= 1,
  \]
as \(\alpha_{1,j+1} = \alpha_{1,j}\).

  Now assume \(\ell \geq 2\). For \(j\in\{1, \dots, k\}\) we can verify that
  \begin{align*}
    c^\ell y^j&=2c^{\ell-1}y^j+d^\ell y^j \\
    &=2c^{\ell-1}y^j+ \underbrace{d^{\ell,1}y^{j,1}}_{= 0, \text{ as }
      d^{\ell,1} = 0} +\underbrace{d^{\ell,2}y^{j,2}}_{=
      j-1}+\underbrace{d^{\ell,3}y^{j,3}}_{\enifed
      \alpha_\ell}+\underbrace{d^{\ell,4}y^{j,4}}_{= 0, \text{ as }
      y^{j,4} = 0 \text{ for } j \leq k}\\
    &=2c^{\ell-1}y^j+(j-1)+\alpha_\ell,
  \end{align*}
  where \(\alpha_\ell =3k\) if \(\ell\) is odd, and otherwise \(\alpha_\ell
  =0\). Thus, for \(j\in\face{1,\dots,k-1}\) we have
  \begin{align*}
    c^\ell y^j-c^\ell y^{j+1}&=2c^{\ell-1}y^j+j-1+\alpha_\ell -(2c^{\ell-1}y^{j+1}+j+\alpha_\ell)\\
    &=2(\underbrace{c^{\ell-1}y^j-c^{\ell-1}y^{j+1}}_{=1, \text{ by induction}})-1 = 1.
  \end{align*}
  We can do a similar analysis for  \(\ell \geq 2\) and \(j\in
  \face{k+1, \dots, 2k}\):
  \begin{align*}
    c^\ell y^j&=2c^{\ell-1}y^j+d^\ell y^j \\
    &=2c^{\ell-1}y^j+ \underbrace{d^{\ell,1}y^{j,1}}_{= 0, \text{ as }
      d^{\ell,1} = 0} +\underbrace{d^{\ell,2}y^{j,2}}_{=
      j-1}+\underbrace{d^{\ell,3}y^{j,3}}_{= 0, \text{ as }
      y^{j,3} = 0 \text{ for } j \geq k+1}+\underbrace{d^{\ell,4}y^{j,4}}_{\enifed
      \beta_\ell}\\
    &=2c^{\ell-1}y^j+(j-1)+\beta_\ell,
  \end{align*}
  where \(\beta_\ell =3k\) if \(\ell\) is even, and \(\beta_\ell
  =0\) otherwise . As before we obtain that for \(j = k+1,\dots,2k-1\),
  \(c^\ell y^j-c^\ell y^{j+1} = 1\) holds. 
\end{proof}

Note that in the above argument the values of \(d^{\ell,3}\), \(d^{\ell,4}\)
are irrelevant as they are eliminated in the difference of two
consecutive points. However, they will become important as they
enable the switching between and linking of the two groups
\(\{y^1, \dots, y^k\}\)
and \(\{y^{k+1}, \dots, y^{2k}\}\)
as we will show now.  To this end we prove the following lemma:

\begin{lemma}[Decreasing intergroup ordering]\label{lem:intergroupOrder}
  For any \(\ell \geq1\), if \(\ell\) is odd then \(c^\ell y^{k}=c^\ell
  y^{k+1}+1\), and if \(\ell\) is even then \(c^\ell y^{2k}=c^\ell y^1+1\).
\end{lemma}

\begin{proof}
  The proof is by alternating induction on the odd and even case. First
  observe that \(c^1 y^k = c^1y^{k+1} + 1\), which will be the start of our
  induction for the odd case:
  \begin{align*}
    c^1 y^k - c^1y^{k+1} & = \underbrace{d^{1,1} (y^{k,1} -
      y^{k+1,1})}_{ = -(k-1)} + \underbrace{d^{1,2}
      (y^{k,2} - y^{k+1,2})}_{= 0} + \underbrace{d^{1,3} (y^{k,3} -
      y^{k+1,3})}_{= k} +
    \underbrace{d^{1,4} (y^{k,4} - y^{k+1,4})}_{= 0} = 1.
  \end{align*}

  First, let \(\ell \geq 1\) be even
  and suppose \(c^{\ell-1}y^k=c^{\ell-1}y^{k+1}+1\), which is satisfied in the case
  \(\ell =2\) by the above. Then, repeated application of Lemma~\ref{lem:intragroupOrder} yields
  \(c^{\ell-1}y^1=c^{\ell-1}y^{2k}+2k-1\). Moreover, we have
  \begin{align*} c^\ell y^1&=2c^{\ell-1}y^1+d^\ell y^1
    =2c^{\ell-1}y^1+\underbrace{d^{\ell,1}y^{1,1}}_{= 0, \text{ as
        \(\ell > 1\)}} +
    \underbrace{d^{\ell,2}y^{1,2}}_{= 0, \text{ as \(y^{1,2} =
        0\)}}+\underbrace{d^{\ell,3}y^{1,3}}_{= 0, \text{ as \(\ell\) even }} + \underbrace{d^{\ell,4}y^{1,4}}_{= 0, \text{ as \(y^{1,4} =
        0\)}}
    =2c^{\ell-1}y^1,
  \end{align*} and
  \begin{align*} c^\ell y^{2k}&=2c^{\ell-1} y^{2k}+d^\ell y^{2k}
    =2c^{\ell-1}y^{2k}+ \underbrace{d^{\ell,1}y^{2k,1}}_{= 0, \text{ as
        \(\ell > 1\)}} + \underbrace{d^{\ell,2}y^{2k,2}}_{=
      k-1}+\underbrace{d^{\ell,3}y^{2k,3}}_{= 0, \text{ as \(\ell\) even
      }}+\underbrace{d^{\ell,4}y^{2k,4}}_{= 3k} \\
    &=2c^{\ell-1}y^{2k}+(k-1)+3k =2c^{\ell-1}y^{2k}+4k-1.
  \end{align*} Thus, we obtain for the difference
  \begin{align*} c^\ell y^{2k}-c^\ell y^1&=2c^{\ell-1}y^{2k}+4k-1-2c^{\ell-1}y^1\\
    &=2(\underbrace{c^{\ell-1}y^{2k}-c^{\ell-1}y^1}_{= 1 - 2k, \text{
        from above}})+4k-1 =2(1-2k)+4k-1 =1.
  \end{align*}

  Now we consider the case where \(\ell\) is odd, which is similar to the
  one above. Assume that \(c^{\ell-1}y^{2k}=c^{\ell-1}y^1+1\), which we now know to
  hold for \(\ell = 3\) by means of the argument for \(\ell\) even case from
  above.
 Then, applying Lemma~\ref{lem:intragroupOrder} in increasing and
 decreasing direction, we obtain
  \(c^{\ell-1} y^{k} + 2k-1 = c^{\ell-1} y^{k+1}\). We will show that \(c^\ell y^{k}=c^\ell
  y^{k+1}+1\). We have
  \begin{align*}
    c^\ell y^k = 2c^{\ell -1} y^k + \underbrace{d^{\ell,1}y^{k,1}}_{= 0} +
    \underbrace{d^{\ell,2}y^{k,2}}_{= k-1} +
    \underbrace{d^{\ell,3}y^{k,3}}_{=3k}
    +\underbrace{d^{\ell,4}y^{k,4}}_{=0} = 2c^{\ell -1} y^k + 4k -1
  \end{align*}
  and
  \begin{align*}
    c^\ell y^{k+1} = 2c^{\ell -1} y^{k+1} + \underbrace{d^{\ell,1}y^{k+1,1}}_{= 0} +
    \underbrace{d^{\ell,2}y^{k+1,2}}_{= 0} +
    \underbrace{d^{\ell,3}y^{k+1,3}}_{=0}
    +\underbrace{d^{\ell,4}y^{k+1,4}}_{=0} = 2c^{\ell -1}y^{k+1},
  \end{align*}
  so that
  \begin{align*}
    c^\ell y^k - c^\ell y^{k+1} = 2 (\underbrace{c^{\ell - 1}y^k -
      c^{\ell-1} y^{k+1}}_{=1-2k}  ) + 4k -1 = 1.&\qedhere
  \end{align*}
\end{proof}

With these
last two lemmas in hand, we are ready to prove the worst-case lower
bound. The proof describes the possible behavior of the bit scaling
algorithm when given a polytope \(P_n\) and cost vector \(c^p\), as depicted in Figure \ref{fig:worstCase}. The
\(\Omega(n\log\mnorm{c^p})\) lower bound proven here meets the upper
bound established in Lemma~\ref{lem:bitScaling}, implying that the
analysis is tight.

\begin{figure}
\begin{center}
\begin{tikzpicture}[scale=.75, transform shape]
\tikzstyle{circlenode}=[draw,circle,minimum size=1.35cm]
\node [circlenode] (y11) at (12.5,10.5) {\(y^1\)};
\node [circlenode] (y21) at (10.5,10.5) {\(y^2\)};
\node [circlenode] (yk-11) at (7.5,10.5) {\(y^{k-1}\)};
\node [circlenode] (yk1) at (5.5,10.5) {\(y^k\)};
\node [circlenode] (yk+11) at (3.5,10.5) {\(y^{k+1}\)};
\node [circlenode] (yk+21) at (1.5,10.5) {\(y^{k+2}\)};
\node [circlenode] (y2k-11) at (-1.5,10.5) {\(y^{2k-1}\)};
\node [circlenode] (y2k1) at (-3.5,10.5) {\(y^{2k}\)};
\node (ell11) at (9,10.5) {\(\cdots\)};
\node (ell21) at (0,10.5) {\(\cdots\)};
\node (blank11) at (0,11) {};
\node (blank12) at (9,11) {};
\draw [-triangle 60] (y2k1)edge[bend left=60](y2k-11);
\draw [-triangle 60] (y2k-11)edge[bend left=60](blank11);
\draw [-triangle 60] (blank11)edge[bend left=60](yk+21);
\draw [-triangle 60] (yk+21)edge[bend left=60](yk+11);
\draw [-triangle 60]  (yk+11) edge[bend left=60] (yk1);
\draw [-triangle 60]  (yk1) edge[bend left=60] (yk-11);
\draw [-triangle 60]  (yk-11) edge[bend left=60] (blank12);
\draw [-triangle 60]  (blank12) edge[bend left=60] (y21);
\draw [-triangle 60]  (y21) edge[bend left=60] (y11);
\node [circlenode] (y11) at (12.5,8) {\(y^1\)};
\node [circlenode] (y21) at (10.5,8) {\(y^2\)};
\node [circlenode] (yk-11) at (7.5,8) {\(y^{k-1}\)};
\node [circlenode] (yk1) at (5.5,8) {\(y^k\)};
\node [circlenode] (yk+11) at (3.5,8) {\(y^{k+1}\)};
\node [circlenode] (yk+21) at (1.5,8) {\(y^{k+2}\)};
\node [circlenode] (y2k-11) at (-1.5,8) {\(y^{2k-1}\)};
\node [circlenode] (y2k1) at (-3.5,8) {\(y^{2k}\)};
\node (ell11) at (9,8) {\(\cdots\)};
\node (ell21) at (0,8) {\(\cdots\)};
\node (blank11) at (0,8.5) {};
\node (blank12) at (9,8.5) {};
\draw [red, -triangle 60] (y11)edge[out=225,in=320,looseness=.25](y2k1); 
\draw [-triangle 60] (y2k1)edge[bend left=60](y2k-11);
\draw [-triangle 60] (y2k-11)edge[bend left=60](blank11);
\draw [-triangle 60] (blank11)edge[bend left=60](yk+21);
\draw [-triangle 60] (yk+21)edge[bend left=60](yk+11);
\node [circlenode] (y11) at (12.5,5.5) {\(y^1\)};
\node [circlenode] (y21) at (10.5,5.5) {\(y^2\)};
\node [circlenode] (yk-11) at (7.5,5.5) {\(y^{k-1}\)};
\node [circlenode] (yk1) at (5.5,5.5) {\(y^k\)};
\node [circlenode] (yk+11) at (3.5,5.5) {\(y^{k+1}\)};
\node [circlenode] (yk+21) at (1.5,5.5) {\(y^{k+2}\)};
\node [circlenode] (y2k-11) at (-1.5,5.5) {\(y^{2k-1}\)};
\node [circlenode] (y2k1) at (-3.5,5.5) {\(y^{2k}\)};
\node (ell11) at (9,5.5) {\(\cdots\)};
\node (ell21) at (0,5.5) {\(\cdots\)};
\node (blank11) at (0,6) {};
\node (blank12) at (9,6) {};
\draw [red, -triangle 60]  (yk+11) edge[bend left=60] (yk1); 
\draw [-triangle 60]  (yk1) edge[bend left=60] (yk-11);
\draw [-triangle 60]  (yk-11) edge[bend left=60] (blank12);
\draw [-triangle 60]  (blank12) edge[bend left=60] (y21);
\draw [-triangle 60]  (y21) edge[bend left=60] (y11);
\node(label) at (-6.5,10.5) {\large Optimize over \(c^1\)};
\node(label) at (-6.5,8) {\large Optimize over \(c^2\)};
\node(label) at (-6.5,5.5) {\large Optimize over \(c^3\)};
\draw[-triangle 60](-6.5,10)--(-6.5,8.5);
\draw[-triangle 60](-6.5,7.5)--(-6.5,6);
\draw[-triangle 60](-6.5,5)--(-6.5,3.5);
\node(text) at (-6.5,3) {\(\vdots\)};
\end{tikzpicture}
\caption{Points visited by the bit scaling algorithm in the worst
  case. Black arcs follow via Lemma~\ref{lem:intragroupOrder}, red
arcs via Lemma~\ref{lem:intergroupOrder}.}
\label{fig:worstCase}
\end{center}
\end{figure}
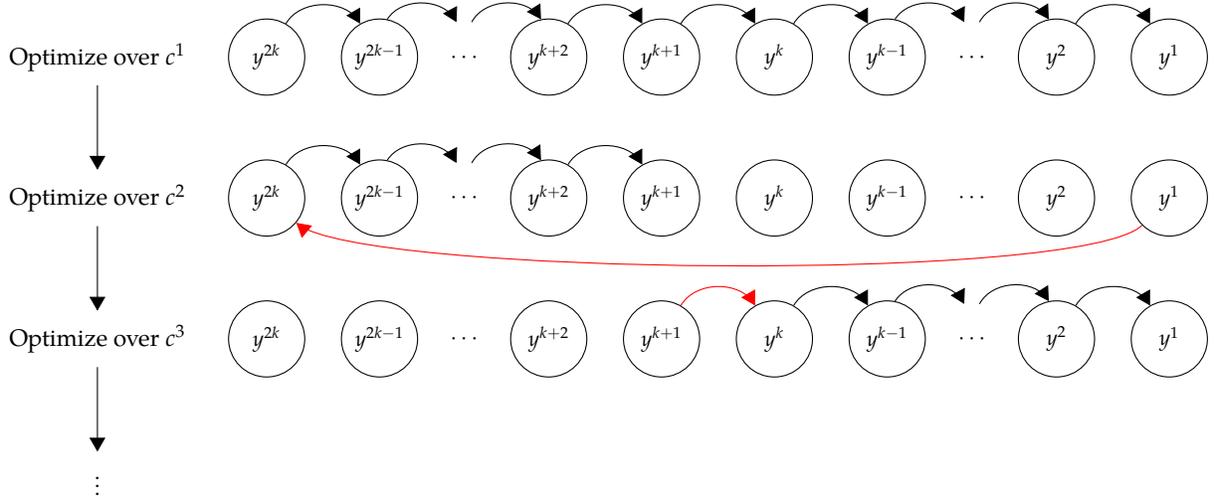

\begin{theorem}\label{thm:lbBS}
  Choose \(k\geq1\) and set \(n \define 8k-2\). Let
  \(P_n=\conv{\face{y^1,\dots,y^{2k}}}\) be the polytope and \(c^p\) for some \(p\geq1\) the objective function as constructed
  above. Then the bit scaling algorithm optimizing \(c^p\) over \(P_n\) requires \(\Omega(n\log\mnorm{c^p})\) augmentation steps in the worst case.
\end{theorem}

\begin{proof}
  By construction of \(c^p\), the bit scaling algorithm optimizes over \(c^1,c^2, \dots, c^p\) in successive scaling phases. The algorithm begins by optimizing over \(c^1\). By Lemmas~\ref{lem:intragroupOrder} and \ref{lem:intergroupOrder} we have
  \[
  c^1y^{2k}<c^1y^{2k-1}<\cdots<c^1y^1.
  \]
  Since an augmentation step moves to any point with improving cost, the algorithm may be forced to visit all \(2k\) points when optimizing over
  \(c^1\).

  For \(\ell \geq 2\) and \(\ell\) even, \(y^1\) maximizes
  \(c^{\ell-1}\) over \(P_n\) and
  \[
  c^\ell y^1<c^\ell y^{2k}<c^\ell y^{2k-1}<\cdots<c^\ell y^{k+1},
  \]
  so the bit scaling algorithm may visit all \(k\) points in
  \(\face{y^{k+1},\dots,y^{2k}}\) in the \(\ell\)th scaling phase. Similarly, for \(\ell \geq 2\) and \(\ell\) odd,
  \(y^{k+1}\) maximizes \(c^{\ell-1}\) over \(P_n\) and
  \[
  c^\ell y^{k+1}<c^\ell y^k<c^\ell y^{k-1}<\cdots<c^\ell y^1,
  \]
  so the algorithm may visit all \(k\) points in \(\face{y^1,\dots,y^k}\). Thus,
  for \(\ell\in\{ 1, \dots, p\}\), at least \(k\) augmentations may be necessary to optimize over \(c^\ell\). As \(p=\ceil{\log\mnorm{c^p}}\), this gives a total number of (at least)
  \[ 
  k\ceil{\log\mnorm{c^p}}=\frac{n+2}{8}\ceil{\log\mnorm{c^p}}\in\Omega(n\log\mnorm{c^p})
  \]
  augmentations necessary over the entire algorithm.
\end{proof}

Observe that in the example we have constructed above, we revisit the
points \(y^j\) several times, which leads to the high worst-case number of augmentations. However, given the same polytope/cost vector pair \((P_n,c^p)\), geometric scaling behaves differently. In particular, as shown in Observation~\ref{obs:GSnoRevisit}, the geometric scaling algorithm never revisits a point. Thus the number of augmentations necessary for geometric scaling is bounded by the number of vertices of \(P_n\), which is~\(\order{n}\). With a suitable choice of \(p\)  (recall \(\log\mnorm{c^p}\approx p\)), the number of augmentations calculated by the two methods can have an arbitrarily high difference. We thus have the following corollary.

\begin{corollary}\label{cor:bitscale:worstcase}
  For any \(p \geq 1\),
  there exists a polytope \(P \subseteq \cube{n}\)
  with \(n=8k+2\), \(k\in\Z_{>0}\)
  and an objective function \(c = c^p\), so that bit scaling computes
  \(\Omega(n \log\mnorm{c^p})=\Omega(np)\)
  augmenting directions in the worst case, while geometric scaling
  needs \(\order{n}\)
  augmenting directions. In particular, the relative difference can be 
  made arbitrarily large by choosing \(p\) appropriately.
\end{corollary}

Theorem~\ref{thm:lbBS} is particularly
interesting as it shows that the number of required augmentations for the bit scaling
algorithm is unbounded for 0/1 polytopes. On the other hand, the rounding scheme of \cite{frank1987application} can be used to turn an arbitrary \(c \in \Q^n\) into a vector \(\bar c \in \Z^n\) with encoding length \(O(n^3)\) in time polynomial in \(n\) and \(\log \mnorm{c}\) such that optimizing both vectors results in the same optimal solution. Thus, bit scaling requires at most \(O(n^4)\) augmentations in the worst-case \emph{with} preprocessing of the objective function. We obtain the same worst-case bound on the number of augmentations for geometric scaling.  

\section{Implementation}
\label{sec:implementation}

We implemented the discussed algorithms bit scaling
(Algorithm~\ref{alg:bitScaling}), MRA (Algorithm~\ref{alg:mra}), and geometric scaling
(Algorithm~\ref{alg:mraScaling}) in C using
the framework SCIP, see~\cite{Ach2009,SCIP}. We also implemented a simple
augmentation algorithm (``augment'') that iteratively searches for
augmenting directions. All methods run for arbitrary mixed-integer problems
as described in the following sections. Generally, for each instance we run
presolving and solve the root node, including cuts. We then start running
the described algorithms, but keep all cuts and solutions found by
heuristics so far, including those found during the augmentation iterations.

\subsection{Solving the augmentation problems}

The augmentation problem is solved as a MIP. In general, we add an
objective cut \(c x \geq c x^k + \delta\) with respect to the current
objective~\(c\) and last feasible iteration point~\(x^k\). We use \(\delta
=\ceil{2 \varepsilon \cdot \abs{c x^k}}\) if the objective is
\emph{integral}, i.e., is guaranteed to yield integral values for all
feasible solutions; we set \(\delta = 2 \varepsilon\cdot \abs{c x^k}\)
otherwise. Here \(\varepsilon = 10^{-6}\) is the feasibility tolerance of
SCIP. (The first solution is found without adding this constraint.)  Note
that continuous variables do not need to be treated differently, i.e., the
approach works for arbitrary MIPs.

We then solve the MIP subproblem until we find an improving
solution~\(x^{k+1}\). For any such solution, we try to exhaust the
direction, by searching for the largest integral \(\alpha\) such that \(x^k
+ \alpha(x^{k+1} - x^k)\) is feasible.

The search for improving solutions can be incomplete: We first solve the
root node of the subproblem and check whether we found an improving solution. If yes, we use
this solution as an augmentation direction. Otherwise, we continue to solve
the MIP until we find any feasible solution. It often happens that soon
after finding some feasible solution, further solutions are found, e.g.,
by so-called exchange heuristics like 1-opt or crossover. We therefore
continue the solution process
until for a fixed number of nodes no further improving solution is found or the
problem has been solved (this is called a ``stall node limit'' in SCIP).

Note that the last iteration has to be solved to optimality in all
algorithms.

\subsection{Augment}
\label{sec:impl-augment}

For the basic augmentation method we proceed as follows: at each iteration
with current best solution~\(x^k\), we add an objective cut. We then
iteratively search for improving solutions until we prove infeasibility (in
this case, \(x^k\) is optimal) or hit the time limit. Note that we solve
the same problem as the original with an additional constraint. Since this
constraint does not cut off any better solution, the dual bound obtained in
each subproblem is valid for the original.

\subsection{Bit scaling}
\label{sec:impl-bitscaling}

\begin{algorithm}[tb]
  \caption{\label{alg:bitScalingMip}Bit Scaling Variant}
  \textbf{Input: } Feasible solution \(x_0\)\\
  \textbf{Output: } Optimal solution for \(\max \set{cx}{x \in P \cap \Z^n}\)
  \begin{algorithmic}
    \State \(\mu \leftarrow 2^{\ceil{\log C}}\)
    \Repeat
    \State \textbf{set} \(c_0 \leftarrow \floor{c / \mu}\)
    \State \textbf{compute} \(x_0 \in P\) integral with \(c_0 x_0 = \max\face{ cx
      \mid x \in P_I}\)
    \State \(\mu \leftarrow \mu/2\)
    \Until{\(\mu < 1\)}
    \State \Return \(x_0\) \Comment{return optimal solution}
  \end{algorithmic}
\end{algorithm}

In the bit scaling Algorithm~\ref{alg:bitScaling}, at each iteration, the
problem is solved with the scaled objective function under the additional
constraint that a solution must improve on the current one. The scaling
factor is changed if no improving solution exists. In Algorithm
\ref{alg:bitScalingMip}, there is no additional constraint, but the optimal
solution is computed at each iteration, rather than an improving solution;
therefore, the scaling factor changes at each iteration. We have implemented
both algorithms, as well as a variant of Algorithm \ref{alg:bitScalingMip}
with an improving constraint (similar to Algorithm
\ref{alg:bitScaling}). This may alter the behavior of the MIP solver, but
not the behavior of the algorithm itself.

Let us now point out further implementation issues. At the beginning of the
algorithms, the objective function is replaced with the scaled version.  In
practice, the coefficients of the objective function may not be integer.
Thus, an additional scaling may be needed to make the objective integral.
For the variants where an improving constraint is used, such a constraint
is added at the beginning of an iteration if a feasible solution is known.
Between two iterations, we compare the objective functions and only solve
the next iteration if the vectors are not equal up to a factor. Depending
on the algorithm, we solve the iteration problem to optimality or use an incomplete
search, as explained above. A new scaling
factor is computed before each new iteration if the solution of the current
iteration is optimal for the current objective function (note that this is
systematically the case for Algorithm \ref{alg:bitScalingMip} and its
variant).

\subsection{Geometric scaling}
\label{sec:Implementation:Geometric}

Geometric scaling is implemented as described in
Algorithm~\ref{alg:mraScaling}, using the function
\begin{equation}\label{eq:practicalPotential}
  \rho(\tilde{x}, x - \tilde{x}) \define \onorm{x - \tilde{x}},
\end{equation}
i.e., \(\rho(\tilde{x},z) = \onorm{z}\). Note that for 0/1 problems this
function is equal to \(\abs{\supp(x - \tilde{x})}\), which is a potential
function, see Section~\ref{sec:geometric-scaling}. For general integer
variables, \(\rho\) does not fulfill Part 1 of
Definition~\ref{def:potential}. We nevertheless use this function, since it
induces sparsity, is easy to handle, and also can be used for the MRA algorithm (see
Section~\ref{sec:Implementation:MRA}). Note that for general integer
variables, we need to add artificial (continuous) variables that model the positive
and negative parts.

Deviating from Algorithm~\ref{alg:mraScaling}, we start
with \(\mu\) equal to the smallest power of 2 larger than the value of
any previously found solution; moreover, we make sure that the value of
\(\mu\) is not larger than \(10^8\). If the objective is not integral, we may
need to solve one final problem, if \(\mu < 1 / n\). For each found
solution, we try to exhaust the direction as described above.

\subsection{Primal heuristic based on geometric scaling}
\label{sec:Implementation:GeometricHeuristic}

It will turn out in our computational results that geometric scaling in
general performs quite well. This motivates the implementation of a primal
heuristic based on it. This heuristic is activated during an ordinary
branch-and-cut run after a primal solution was found and for a certain
number of nodes (by default 200) no further solution was found. It then
runs the geometric scaling algorithm, but with a node limit for the
individual subproblems that depends on the current number of nodes~\(N\). By
default, we use \(\min\{500, \max\{5000, 0.1\, N\}\}\). We also stop if the
total number of nodes exceeds \(0.6\, N\). In this way, the effort spent
in this heuristic is limited, and one can still benefit from solutions found
during the ordinary tree search.

\subsection{MRA}
\label{sec:Implementation:MRA}

The implementation of MRA (Algorithm~\ref{alg:mra}) is based on
\(\rho(\tilde{x}, x - \tilde{x}) = \onorm{x - \tilde{x}}\), as well. Note
that this function is convex in~\(x\). We then want to solve
\[
\max \face{ \frac{c(x - \tilde{x})}{\rho(\tilde{x}, x - \tilde{x})}
  \mid c(x - \tilde{x}) > 0,\; x \in P,\; x \text{ integral}}
\]
where \(\tilde{x}\) is some feasible solution. To solve this fractional
program, we introduce a parameter \(\mu \geq 0\) and check whether \(c(x -
\tilde{x}) \geq \mu\cdot \rho(\tilde{x}, x - \tilde{x})\) by maximizing \(c(x
- \tilde{x}) - \mu\cdot \rho(\tilde{x}, x - \tilde{x})\), which is concave
in~\(x\). We then perform a binary search over \(\mu\), increasing \(\mu\)
if the objective value is positive and decreasing \(\mu\) otherwise. To solve the inner
optimization problem of maximizing \(c(x - \tilde{x}) - \mu\,
\rho(\tilde{x}, x - \tilde{x})\), we rewrite it as
\[
\max \face{ cx - \tau \mid \tau \geq c\tilde{x} + \mu\cdot \rho(\tilde{x}, x - \tilde{x}),\;
c(x - \tilde{x}) > 0,\; x \in P,\; x \text{ integral}}.
\]
(Note that: \(cx - \tau \leq cx - c\tilde{x} - \mu\cdot \rho(\tilde{x}, x - \tilde{x}) = c(x -
\tilde{x}) - \mu\cdot \rho(\tilde{x}, x - \tilde{x})\) for all feasible~\(x\).)

We solve this problem by iteratively generating subgradients for the
convex function
\[
f_\mu(x) \define c\tilde{x} + \mu\cdot \rho(\tilde{x}, x - \tilde{x}).
\]
Its subdifferential is
\[
\partial f_\mu(x) = c - \mu\, \partial \onorm{\cdot}(x - \tilde{x}) = c - \mu \sgn(x - \tilde{x}),
\]
where we define the set of vectors
\[
\sgn(x)_j \define
\begin{cases}
  \{1\} & \text{if } x_j > 0\\
  [-1,1] & \text{if } x_j = 0\\
  \{-1\} & \text{if } x_j < 0
\end{cases}
\qquad\text{ for all }j = 1, \dots, n.
\]
For each subgradient \(h \in \partial f_\mu(\tilde{x})\), we obtain the
subgradient inequality
\[
f_\mu(x) \geq f_\mu(\tilde{x}) + h(x - \tilde{x}).
\]
Assuming that we have generated subgradients \(h^1, \dots, h^k\) for points
\(x_1, \dots, x_k\), we solve
\begin{equation}\label{eq:MRAinner}
\max \face{ cx - \tau \mid \tau \geq f_\mu(x_i) + h^i(x - x_i),\; i = 1,
  \dots, k,\; c(x - x_0) > 0,\; x \in P,\; x \text{ integral}}.
\end{equation}
Let the optimal solution be~\(x^{k+1}\). We then compute a subgradient at
\(x^{k+1}\) and check whether its subgradient inequality is violated. Note
that each solution of~\eqref{eq:MRAinner} is a primal solution that can be stored and used for the
original problem.

[In our implementation, we actually solve the minimization version
\[
\min \face{ -cx + \tau \mid \tau - h^i x \geq f_\mu(x_i) - h^i x_i,\; i = 1,
  \dots, k,\; c(x - x_0) > 0,\; x \in P,\; x \text{ integral}},
\]
for technical reasons.]

In the first iteration, the problem is unbounded, since \(\tau\) is
unbounded. We therefore start with~\(x^0\) (which is actually infeasible)
and take the subgradient \(h^0 = c \in \partial f_\mu(x^0)\). The
inequality added to the optimization problem is then
\[
\tau \geq c x^0 + c(x - x^0) = c x.
\]
Thus, \(\tau\) is bounded from below, if there exists a feasible \(x \in
P\), \(c(x - x^0) > 0\), \(x\) integral.

\section{Computational results}
\label{sec:comp-results}

The algorithms were tested on a Linux cluster with 3.2 GHz Intel i3
processors with 8 GB of main memory and 4 MB of cache, running a single
process at a time. We use SCIP 3.2.0 and CPLEX 12.6.1 as the LP-solver.
SCIP runs with default settings, except that we turn off the ``components''
presolver, since it would decompose the problem into several runs, making a
comparison more difficult.

We use the following testsets:
\begin{description}
\item[MIPLIB2010] The 87 benchmark instances from MIPLIB 2010\footnote{available at
    \url{http://miplib.zib.de/}}, see~\cite{Kochetal2011}.
\item[LB] We use the testset of 29 instances from the ``local branching''
  paper\footnote{available at
    \url{http://www.or.deis.unibo.it/research_pages/ORinstances/MIPs.html}},
  see ~\cite{FisL03}. This testset has also been used
  in~\cite{HanMU06}. The latter paper also contains improved results, which
  we will use below.
\item[QUBO] We use a testset of linearizations of 50 instances for quadratically unconstrained Boolean
  optimization (QUBO)\footnote{available at
    \url{http://researcher.watson.ibm.com/researcher/files/us-sanjeebd/chimera-data.zip}},
  see \cite{dash2013note,dash2015optima}.
\end{description}

In an online supplement, we present details of the computations described
in the following. We will first discuss results on the MIPLIB2010 test
set. As it turns out, augmentation methods do not help to solve these
instances, essentially because they are too easy. We then consider the very
hard testsets LB and QUBO. Here, it will turn out that augmentation methods
significantly improve on the default settings and produce primal solutions
of very good quality.

\subsection{Testset MIPLIB 2010}

\begin{table}
  \caption{Aggregated results of the different algorithms on testset MIPLIB 2010 (1 hour time limit, 87 instances)}
  \label{tab:AggregatedMIPLIB2010}
  \begin{scriptsize}
    \sffamily
    \begin{tabular*}{\textwidth}{@{\extracolsep{\fill}}lrrrrrrrrr@{}}
      \toprule
      name      &  \#nodes &     time &\#run &\#best &\#improv. & \#subprob. & \#phases & \#exhaust & prim-$\int$\\
      \midrule
      bitscale              &  20116.2 &   934.93 & 59 & 67 &   3.7 &   8.4 &   4.8 &   0.0 &     57.4\\
      MRA                   &  16434.8 &  1819.46 & 83 & 32 & 411.4 & 428.6 & 418.7 &   8.1 &    207.2\\
      geometric             &   5628.8 &  1632.31 & 83 & 65 &   6.2 &  23.6 &  17.4 &   0.0 &     49.8\\
      augment               &  14793.5 &   924.71 & 83 & 63 &  12.9 &  12.9 &  12.9 &   0.0 &     54.5\\
      bitscale-classic      &  25086.0 &  1120.62 & 59 & 62 &   4.1 &  11.0 &   6.9 &   0.0 &     60.4\\
      bitscale-noimprove    &  20343.4 &   918.31 & 60 & 69 &   4.2 &   8.8 &   4.6 &   0.0 &     56.7\\
      bitscale-complete     &  26902.6 &  1070.71 & 59 & 59 &   2.0 &   6.6 &   4.6 &   0.0 &    103.5\\
      geom-no $\ell_1$      &   6038.5 &  1651.26 & 83 & 65 &   6.4 &  23.6 &  17.2 &   0.0 &     49.5\\
      geom-no cutoff        &   7648.7 &  2062.03 & 83 & 49 &   3.9 &  19.4 &  15.6 &   0.0 &     85.4\\
      geom-8                &   8637.7 &  1313.55 & 83 & 69 &   5.7 &  12.8 &   7.1 &   0.0 &     40.7\\
      geom-64               &   8680.9 &  1122.85 & 83 & 72 &   6.4 &  10.7 &   4.3 &   0.0 &     39.7\\
      geom-256              &   8358.2 &  1128.40 & 83 & 69 &   7.1 &  10.9 &   3.7 &   0.0 &     39.7\\
      geom-512              &   9895.4 &  1112.03 & 83 & 69 &   7.0 &  10.3 &   3.3 &   0.0 &     41.8\\
      geom-1024             &   8392.5 &  1016.80 & 83 & 71 &   7.1 &  10.3 &   3.2 &   0.0 &     39.8\\
      \midrule
      geom-heur             &  16858.8 &   741.52 & 68 & 71 &   1.3 &  33.6 &  32.4 &   0.0 &     30.4\\
      geom-infer            &  21343.6 &   732.30 & 68 & 72 &   1.0 &  41.7 &  40.7 &   0.0 &     30.1\\
      geom-heur-64          &  16158.0 &   683.18 & 68 & 73 &   1.7 &   9.9 &   8.1 &   0.0 &     29.3\\
      \midrule
      default               &  15495.9 &   557.12 &  0 & 74 &   0.0 &   0.0 &   0.0 &   0.0 &     26.3\\
      \bottomrule
    \end{tabular*}
  \end{scriptsize}
\end{table}

Table~\ref{tab:AggregatedMIPLIB2010} shows a comparison of the four
different augmentation methods and variants of these as well as for the
default settings on the testset MIPLIB2010. In the table, ``\#nodes'' and
``time'' give the shifted geometric means\footnote{The shifted geometric
  mean of values \(t_1, \dots, t_n\) is defined as \(\big(\prod(t_i +
  s)\big)^{1/n} - s\) with shift \(s\). We use a shift \(s = 10\) for time
  and \(s = 100\) for nodes in order to decrease the strong influence of
  the very easy instances in the mean values.} of the total number of nodes
(including subproblems) and the time (in seconds), respectively. Column
``\#run'' presents the number of instances for which an augmentation
routine ran. Column ``\#best'' refers to the number of times the best known
primal solution value has been found. With respect to the augmentation
methods, the columns ``\#improv.'', ``\#subprob.'', ``\#phases'', and
``\#exhaust'' refer to the average number of times an improved primal
solution has been found, the number of subproblems (MIPs) solved, the
number of phases, and the number of exhausting directions found,
respectively. The number of phases refers to the number of subproblems
solved with the same value of \(\mu\) for bit and geometric scaling (in
this case, \(\text{\#phases} + \text{\#improv} = \text{\#subprob}\)); note
that we count a possible search for the first primal solution as one
phase. For MRA, we count each outer iteration as a phase; thus, the number
of phases equals the number of Benders problems~\eqref{eq:MRAobj}
solved. For \textsf{augment}, the number of phases equals the number of
improving solutions and the number of subproblems.

Finally, the last column gives the \emph{primal integral},
see~\cite{Berthold13}. The primal integral is the value we obtain by
integrating the gap between the current primal and best primal bound over
time\footnote{We define the gap between primal bound~\(p\) and best primal
  bound~\(b\) as \(\abs{p-b}/\max(\abs{p},\abs{b})\).}. Thus, a smaller primal
integral indicates a higher solution quality over time.

We can draw several conclusions from these experiments:

\paragraph{General Observations}

Among the 87 instances, four were solved before the end of the root node
and no augmentation routine was applied.

The number of phases and subproblems is usually below 30, with the
exception of \textsf{MRA}, which needs a larger number of phases and
subproblems. Thus, it seems that in practice no long series of
augmentation steps occur, but as the example of MRA shows, this would in
principle be possible. Note, however, that the numbers are significantly
smaller than the theoretical bounds.

The goal of the table is to illustrate the behavior of the augmentation
procedures. However, we also add the results of the default settings for
comparison. It turns out---as possibly expected---that these settings are
faster and have a smaller primal integral than all stand-alone augmentation
procedures on average.

\paragraph{Bitscaling}
Bitscaling only runs on 59 of the 87 instances, since the
other instances have equal nonzero objective coefficients. Note
that this somewhat reduces the corresponding averages, since the default
settings are often faster.

We compare four variants of bit scaling (\textsf{bitscale},
\textsf{bitscale-classic}, \textsf{bitscale-noimprove}, and
\textsf{bitscale-complete}). The basic variant (\textsf{bitscale}) uses
incomplete searches and an improving constraint. Recall that when
performing incomplete searches, only some improving solution is searched
for at each iteration, as opposed to the optimal solution. The objective
function changes at each iteration.

The \textsf{classic} variant corresponds to Algorithm \ref{alg:bitScaling}:
each improving solution requires a solve, and the objective function is
scaled only when the problem is infeasible (due to the improving
constraint).  This explains why this variant has the highest number of
phases among all bit scaling variants. The \textsf{classic} variant
outperformed the \textsf{bitscale} variant on only three instances out of
the 59.

The \textsf{noimprove} variant is similar to \textsf{bitscale}, except that
no improving constraint is added. The results are very close to
\textsf{bitscale}, maybe slightly better. Depending on the instance one variant can
be significantly faster than the other: indeed, \textsf{bitscale} can be
3.5 times as fast as \textsf{noimprove}, but also two times as slow. Note,
however, that this might be the result of performance variability
(see~\cite{Kochetal2011}).

In the \textsf{complete} variant, each phase is solved to optimality.
This variant thus requires fewer phases than \textsf{bitscale}.
However, on this testset, the \textsf{complete} variant is never faster
than \textsf{bitscale} and can perform up to 4.8 times slower.

Let us now compare the default \textsf{bitscale} variant to default
SCIP. It is faster on 10 instances out of the 59 on which \textsf{bitscale}
runs. While this shows that in some cases using bit scaling can be
beneficial, default SCIP performs overall significantly better.

\paragraph{MRA}
For MRA, the difference between the number of subproblems and the number of
phases is surprisingly small. This indicates that only very few subgradients are
needed in MRA to compute the optimal inner value in~\eqref{eq:MRAinner}: on
average at most two subgradients are added in each phase. Note also that MRA generates a large number of improving
solutions. Obviously, MRA takes a disadvantageous route through the feasible
solutions.

When comparing the different variants, MRA is clearly the slowest, solves
the fewest number of instances, and uses the largest number of phases. We
currently do not have an explanation for this large difference.

Interestingly, MRA is the only variant for which the exhausting step
actually was performed. Obviously the solutions produced by the other
variants are always automatically exhaustive.

\paragraph{Geometric Scaling}
Next to the default settings, variant \textsf{geometric} solves the largest
number of instances. It is faster than the default for two
instances. However, it uses quite a number of phases and subproblems
without finding an improving solution, in particular at the beginning.

It turns out that not using the \(\ell_1\)-norm on general integer
variables (variant \textsf{geom-no \(\ell_1\)}), i.e., we ignore these variables in the
objective function, does not make a difference. This can be explained by
the fact that on average there are only a few general integer variables.

Interestingly, adding an objective constraint instead of using an
objective cutoff (variant \textsf{geom-no cutoff}), worsens the
results significantly. This might be due to the fact that SCIP stores
suboptimal solutions and uses them to generate better ones, while
infeasible solutions are not stored.

Finally, we consider different factors to update \(\mu\) in
Algorithm~\ref{alg:mraScaling} (the default factor is 2). When increasing
this factor, the running time generally decreases. The largest number of
solutions with objective equal to the best value and the smallest primal
integral appear for a factor of 64. The corresponding variant
\textsf{geom-64} is better than the default for three instances.

\paragraph{Augment}
Compared to the other variants, \textsf{augment} is surprisingly fast. It
also usually takes few phases, but produces more improving solutions
(with the exception of MRA). However, \textsf{bitscale} is not far
behind. Moreover, most geometric scaling variants produce better solutions
(\#best) and smaller primal integrals, but they use more time on average.

\paragraph{Geometric scaling heuristic}
We also tested the heuristic based on geometric scaling (see
Section~\ref{sec:Implementation:GeometricHeuristic}), with the following
settings: \textsf{geom-heur}, \textsf{geom-heur-infer}, and
\textsf{geom-heur-64}. Here, \textsf{geom-heur-infer} uses inference
branching instead of the default branching rule, which should lead to
reduced times for the branching rules. Moreover, \textsf{geom-heur-64} uses a
factor of~64 to reduce \(\mu\), since this produced the best results for
geometric scaling.

The running times and number of nodes are larger than the default. If we
only consider the number of nodes in the main branch-and-bound tree on
instances which were solved to optimality, variants \textsf{geom-heur} and
\textsf{geom-heur-64} reduce the number of nodes by about 10\,\% and 13\,\%
relative to the default settings, respectively. However, this improvement is over-compensated by
the overhead incurred by the heuristic. Moreover, because of the node
limits, a larger number of subproblems could be treated, but the number of
improving solutions is smaller relative to the other methods.

In total, we conclude that the heuristic does not help to improve
the performance on the MIPLIB 2010 instances -- essentially, most of these
instances are ``too easy''.
\smallskip

As a general comment, note that the influence of heuristics on the
performance is generally not too large: \cite{Berthold14} estimated the
difference of the running time of SCIP using heuristics and not using any
heuristics to be about 11\,\% on the MIPLIB 2010 benchmark
testset. Moreover, a single heuristic has the disadvantage to ``compete''
against the other heuristics (42 in SCIP -- not all of them active). On the
other hand, the augmentation methods significantly benefit from good
heuristics if an incomplete search is used.

\subsection{LB testset}
   \IfFileExists{results/results-lb.tex}{
      \afterpage{\begin{landscape}
  \begin{table}
    \caption{Best primal values for different variants on the testset LB (29 instances, 1 hour time limit).
      Column ``previous best'' gives the best value obtained in~\cite{FisL03} or~\cite{HanMU06};
      all problems are minimization instances.
      For each instance, the best values among those obtained by the variants (excluding ``previous best'') are marked in black,
      otherwise the values are marked gray.}
    \label{tab:LB}
    \newcommand{\A}[1]{\textcolor[rgb]{0.6,0.6,0.6}{\num{#1}}}   
    \newcommand{\B}[1]{\num{#1}}                                 
    \newcommand{\C}[1]{\textcolor[rgb]{0.6,0.6,0.6}{\num{#1}}} 
    \begin{scriptsize}
      \sffamily
      \begin{tabular*}{\linewidth}{@{\extracolsep{\fill}}lrrrrrrrr@{}}
        \toprule
        problem        &          default  &          augment  &         bitscale  &          geom-64  &        geom-heur  &  geom-heur-infer  &     geom-heur-64  &    previous best \\
        \midrule
        A1C1S1         & \A{     11643.33} & \A{     11989.36} & \A{     11977.50} & \A{     11638.86} & \B{     11557.22} & \A{     11566.59} & \A{     11590.45} & \B{     11551.19}  \\
        A2C1S1         & \A{     10983.28} & \A{     11422.77} & \A{     11115.34} & \A{     11040.72} & \B{     10897.77} & \A{     10994.27} & \A{     10909.95} & \B{     10889.14}  \\
        arki001        & \B{   7580813.05} & \A{   7581527.87} & \B{   7580813.05} & \A{   7582202.93} &              ---  & \A{   7580814.51} & \A{   7580813.05} & \C{   7580889.44} \\
        B1C1S1         & \B{     24798.51} & \A{     25456.98} & \A{     27309.51} & \A{     25458.30} & \A{     25630.75} & \A{     25123.51} & \A{     25042.56} & \B{     24566.52}  \\
        B2C1S1         & \B{     25763.12} & \A{     27253.74} & \A{     26592.19} & \A{     26167.32} & \A{     26412.44} & \A{     25926.61} & \A{     26002.11} & \C{     26073.78} \\
        biella1        & \B{   3065005.78} & \B{   3065005.78} & \B{   3065005.78} & \B{   3065005.78} & \B{   3065005.78} & \B{   3065005.78} & \B{   3065005.78} & \C{   3070810.15} \\
        core2536-691   & \B{       689.00} & \B{       689.00} & \B{       689.00} & \B{       689.00} & \B{       689.00} & \B{       689.00} & \B{       689.00} & \C{       690.00} \\
        core2586-950   & \A{       970.00} & \A{       972.00} & \A{      1213.00} & \A{       971.00} & \B{       955.00} & \A{       960.00} & \A{       966.00} & \B{       947.00}  \\
        core4284-1064  & \A{      1091.00} & \A{      1100.00} & \A{      3279.00} & \A{      1080.00} & \B{      1072.00} & \A{      1073.00} & \A{      1079.00} & \B{      1065.00}  \\
        core4872-1529  & \A{      1580.00} & \A{      1584.00} & \A{      1769.00} & \A{      1579.00} & \B{      1546.00} & \A{      1560.00} & \A{      1575.00} & \B{      1534.00}  \\
        danoint        & \B{        65.67} & \B{        65.67} & \B{        65.67} & \B{        65.67} & \B{        65.67} & \B{        65.67} & \B{        65.67} & \B{        65.67}  \\
        glass4         & \A{1600013500.00} & \A{1500014200.00} & \A{2200016050.00} & \A{1620014440.00} & \B{1500012650.00} & \A{1550012462.72} & \A{1566683416.66} & \B{1400013666.50}  \\
        markshare1     & \B{         7.00} & \A{         9.00} & \A{        32.00} & \A{        12.00} & \A{        10.00} & \A{        10.00} & \A{        10.00} & \B{         7.00}  \\
        markshare2     & \A{        12.00} & \A{        13.00} & \A{       128.00} & \A{        17.00} & \A{        14.00} & \A{        14.00} & \B{        10.00} & \C{        14.00} \\
        mkc            & \A{      -559.11} & \A{      -542.28} & \A{      -557.56} & \A{      -561.93} & \B{      -562.93} & \A{      -560.85} & \A{      -561.33} & \B{      -563.85}  \\
        net12          & \B{       214.00} & \B{       214.00} & \B{       214.00} & \B{       214.00} & \B{       214.00} & \B{       214.00} & \B{       214.00} & \B{       214.00}  \\
        NSR8K          & \A{ 127262743.24} & \A{  68351187.10} & \A{2176184843.46} & \B{  21415513.00} & \A{ 127262743.24} & \A{ 127262743.24} & \A{ 127262743.24} & \B{  20780430.00}  \\
        nsrand\_ipx    & \B{     51200.00} & \A{     54880.00} & \A{     55200.00} & \A{     52000.00} & \B{     51200.00} & \B{     51200.00} & \B{     51200.00} & \C{     51520.00} \\
        rail507        & \B{       174.00} & \B{       174.00} & \B{       174.00} & \B{       174.00} & \B{       174.00} & \B{       174.00} & \B{       174.00} & \B{       174.00}  \\
        roll3000       & \B{     12890.00} & \A{     12899.00} & \A{     13380.00} & \A{     12904.00} & \B{     12890.00} & \B{     12890.00} & \B{     12890.00} & \B{     12890.00}  \\
        seymour        & \A{       425.00} & \A{       425.00} & \A{       425.00} & \B{       424.00} & \B{       424.00} & \A{       425.00} & \B{       424.00} & \B{       423.00}  \\
        sp97ar         & \A{ 663515230.72} & \A{ 726599877.76} & \A{ 674470726.72} & \B{ 662299239.68} & \A{ 674213859.52} & \A{ 664157022.72} & \A{ 673642038.40} & \C{ 666368944.96} \\
        sp97ic         & \A{ 435258209.12} & \A{ 450307285.28} & \B{ 430937067.04} & \A{ 439446697.12} & \A{ 434570609.44} & \A{ 432663431.84} & \A{ 439022248.00} & \B{ 429892049.60}  \\
        sp98ar         & \A{ 530322047.84} & \A{ 551452928.96} & \A{ 532671408.48} & \B{ 530242941.12} & \A{ 530437736.32} & \A{ 530489389.92} & \A{ 530251516.00} & \C{ 530916867.40} \\
        sp98ic         & \A{ 451409231.04} & \A{ 465544414.56} & \A{ 455081136.48} & \A{ 450843038.08} & \A{ 450519098.72} & \B{ 449226843.52} & \A{ 453626659.52} & \B{ 449226843.52}  \\
        swath          & \A{       494.09} & \A{       502.24} & \A{       506.44} & \A{       495.02} & \B{       467.41} & \A{       481.95} & \A{       477.57} & \B{       467.41}  \\
        tr12-30        & \B{    130596.00} & \B{    130596.00} & \A{    139741.00} & \B{    130596.00} & \B{    130596.00} & \B{    130596.00} & \B{    130596.00} & \B{    130596.00}  \\
        UMTS           & \A{  30094335.00} & \A{  30091967.00} & \B{  30091457.00} & \A{  30092333.00} & \A{  30093479.00} & \A{  30092081.00} & \A{  30091738.00} & \C{  30139634.00} \\
        van            & \B{         5.09} & \A{         5.59} & \A{         5.35} & \A{         6.12} & \B{         5.09} & \B{         5.09} & \B{         5.09} & \B{         4.84}  \\
        \midrule
        \#best:        &               13  &                6  &                8  &               10  &               18  &               10  &               11  \\
        \bottomrule
      \end{tabular*}
    \end{scriptsize}
  \end{table}
\end{landscape}
}
   }
   {
   \IfFileExists{results-lb.tex}{
      \afterpage{}
   }{}
   }

In the next experiment, we compare the results of different augmentation
variants on the testset LB, see Table~\ref{tab:LB}. We use default
settings, \textsf{augment}, and \textsf{bitscale}. Moreover, we apply
geometric scaling with a factor of 64 (\textsf{geom-64}), since this
gave the best results on the MIPLIB2010 testset. Moreover, we again use the
three variants of the geometric scaling heuristic (\textsf{geom-heur},
\textsf{geom-heur-infer}, and \textsf{geom-heur-64}).

The general picture of the augmentation methods is similar to the
MIPLIB2010 testsets; for instance, the number of augmentation subproblems
is generally small (see the online supplement for detailed results). The
results show that variant \textsf{geom-64} dominates \textsf{augment} and
\textsf{bitscale} with respect to the number of instances for which the
best solution among all variants was found. Bit scaling ran for 26 of the
29 instances. Due to the higher number of instances for which it is applied
in comparison to the MIPLIB2010 testset, \textsf{bitscale} performs better
than \textsf{augment}.

The default settings perform very favorably with 14 ``best'' solutions
(better than \textsf{geom-64}), but are dominated by \textsf{geom-heur},
which finds the best value for 18 instances. For the LB testset, it seems
to be essential to have access to the solutions generated by other
heuristics and integer feasible LP solutions during the branch-and-cut
algorithm. These can then be improved by \textsf{geom-heur}. Interestingly,
increasing the \(\mu\) reduction factor in the geometric scaling heuristic
to 64 decreases the number of ``best'' solutions to 11 (see
\textsf{geom-heur-64}). Obviously, the decrease of \(\mu\) is too fast in
order to produce good solutions on this testset. In fact, the number of
phases in \textsf{geom-heur} and \textsf{geom-heur-64} is on average much
larger than for \textsf{geom-64} (the averages are 41.9 for \textsf{geom-heur} and 14.3 for
\textsf{geom-heur-64} vs.\ 4.0 for \textsf{geom-64}). Note that
\textsf{geom-heur} runs out of memory for the instance \textsf{arki001}.

Finally, these results are compared to the best values obtained
in~\cite{FisL03} or~\cite{HanMU06} (presented in column ``previous
best''). These values are improved on nine instances by some variant and on
five by \textsf{geom-heur}. Note that this is not even near a fair
comparison, since the results in \cite{FisL03} and~\cite{HanMU06} were
obtained on different computers, as well as with different implementations
and time limits. Moreover, in the meantime most values have been improved
by other methods. Nevertheless, the results show that using geometric
scaling inside a primal heuristic seems to be promising for obtaining high
quality solutions for hard MIPs.

\subsection{QUBO testset}

\begin{table}
  \caption{Best primal values and primal integral of the default settings and variants of the heuristic based on geometric scaling (50 instances, 1 hour time limit).
    For each instance, the best primal values are marked in black,
    otherwise the values are marked gray; all problems are minimization instances.}
  \label{tab:QUBO}
  \newcommand{\A}[1]{\textcolor[rgb]{0.6,0.6,0.6}{\num{#1}}}   
  \newcommand{\B}[1]{\num{#1}}                                 
  \begin{scriptsize}
    \sffamily
    \begin{tabular*}{\linewidth}{@{\extracolsep{\fill}}lrrrrrrrrrr@{}}
      \toprule
      Problem    & \multicolumn{2}{c}{default} & \multicolumn{2}{c}{geom-64} & \multicolumn{2}{c}{geom-heur} & \multicolumn{2}{c}{geom-heur-infer} & \multicolumn{2}{c}{geom-heur-64}\\
                \cmidrule{2-3}\cmidrule{4-5}\cmidrule{6-7}\cmidrule{8-9}\cmidrule{10-11}
                 &    Primal &  Prim-$\int$ &    Primal &  Prim-$\int$ &    Primal &  Prim-$\int$ &    Primal &  Prim-$\int$ &    Primal &  Prim-$\int$\\
      \midrule
      chim8-4.1  & \A{ -796} & \num{ 153.8} & \A{ -822} & \num{  58.9} & \B{ -830} & \num{ 109.9} & \A{ -788} & \num{ 188.3} & \A{ -798} & \num{ 144.9} \\
      chim8-4.2  & \A{ -776} & \num{ 192.6} & \A{ -766} & \num{ 193.5} & \A{ -794} & \num{  60.5} & \A{ -800} & \num{  33.4} & \B{ -806} & \num{  17.8} \\
      chim8-4.3  & \A{ -784} & \num{ 281.5} & \B{ -840} & \num{ 145.7} & \A{ -790} & \num{ 218.7} & \A{ -790} & \num{ 218.6} & \A{ -800} & \num{ 181.7} \\
      chim8-4.4  & \A{ -806} & \num{ 291.1} & \B{ -876} & \num{  24.7} & \A{ -828} & \num{ 203.5} & \A{ -840} & \num{ 160.0} & \A{ -852} & \num{ 107.2} \\
      chim8-4.5  & \A{ -850} & \num{ 208.1} & \B{ -882} & \num{  58.0} & \A{ -840} & \num{ 175.1} & \A{ -828} & \num{ 223.3} & \A{ -852} & \num{ 128.5} \\
      chim8-4.6  & \A{ -790} & \num{ 218.4} & \A{ -798} & \num{ 189.7} & \B{ -840} & \num{ 140.9} & \A{ -830} & \num{ 158.5} & \A{ -836} & \num{ 132.6} \\
      chim8-4.7  & \A{ -756} & \num{ 301.5} & \A{ -810} & \num{ 134.2} & \A{ -802} & \num{ 102.7} & \B{ -824} & \num{ 188.6} & \A{ -806} & \num{  85.3} \\
      chim8-4.8  & \A{ -786} & \num{ 248.0} & \A{ -822} & \num{  26.2} & \A{ -796} & \num{ 125.8} & \B{ -824} & \num{   4.8} & \A{ -814} & \num{  68.7} \\
      chim8-4.9  & \A{ -850} & \num{ 154.6} & \A{ -810} & \num{ 265.7} & \B{ -872} & \num{ 159.9} & \A{ -802} & \num{ 291.8} & \A{ -828} & \num{ 186.1} \\
      chim8-4.10 & \A{ -810} & \num{ 349.5} & \B{ -896} & \num{  17.7} & \A{ -812} & \num{ 341.1} & \A{ -830} & \num{ 269.3} & \A{ -828} & \num{ 280.2} \\
      chim8-4.11 & \A{ -764} & \num{ 154.6} & \A{ -768} & \num{ 113.3} & \A{ -748} & \num{ 200.7} & \B{ -790} & \num{  68.6} & \A{ -730} & \num{ 280.8} \\
      chim8-4.12 & \A{ -746} & \num{ 306.5} & \A{ -774} & \num{ 161.7} & \A{ -786} & \num{ 105.7} & \A{ -780} & \num{ 130.9} & \B{ -808} & \num{   8.2} \\
      chim8-4.13 & \A{ -818} & \num{ 158.7} & \B{ -848} & \num{  17.8} & \A{ -798} & \num{ 218.4} & \A{ -788} & \num{ 259.6} & \A{ -800} & \num{ 210.5} \\
      chim8-4.14 & \A{ -806} & \num{  63.7} & \A{ -770} & \num{ 218.3} & \B{ -818} & \num{  36.4} & \A{ -800} & \num{ 140.5} & \A{ -792} & \num{ 124.6} \\
      chim8-4.15 & \A{ -836} & \num{ 243.4} & \B{ -896} & \num{  14.9} & \A{ -868} & \num{ 115.9} & \A{ -880} & \num{  67.8} & \A{ -868} & \num{ 118.3} \\
      chim8-4.16 & \A{ -834} & \num{ 164.7} & \A{ -852} & \num{  70.1} & \A{ -814} & \num{ 205.4} & \B{ -862} & \num{  42.0} & \A{ -818} & \num{ 194.2} \\
      chim8-4.17 & \A{ -802} & \num{ 102.7} & \A{ -732} & \num{ 376.1} & \A{ -810} & \num{  81.2} & \B{ -816} & \num{ 157.7} & \A{ -796} & \num{ 114.7} \\
      chim8-4.18 & \A{ -856} & \num{  63.4} & \A{ -808} & \num{ 268.0} & \A{ -868} & \num{  15.2} & \B{ -870} & \num{   5.6} & \B{ -870} & \num{  11.8} \\
      chim8-4.19 & \A{ -870} & \num{ 200.2} & \B{ -906} & \num{  35.3} & \A{ -886} & \num{  84.4} & \A{ -876} & \num{ 126.0} & \A{ -866} & \num{ 164.5} \\
      chim8-4.20 & \A{ -818} & \num{ 249.5} & \A{ -874} & \num{  31.1} & \B{ -878} & \num{  51.5} & \A{ -812} & \num{ 274.6} & \A{ -850} & \num{ 120.3} \\
      chim8-4.21 & \A{ -818} & \num{  98.7} & \A{ -816} & \num{ 120.5} & \A{ -836} & \num{  22.3} & \A{ -830} & \num{  47.6} & \B{ -840} & \num{   5.6} \\
      chim8-4.22 & \A{ -816} & \num{ 122.7} & \A{ -834} & \num{  53.3} & \A{ -820} & \num{ 110.9} & \B{ -844} & \num{  58.7} & \A{ -834} & \num{  48.3} \\
      chim8-4.23 & \A{ -780} & \num{ 196.4} & \B{ -824} & \num{  39.6} & \A{ -788} & \num{ 168.6} & \A{ -768} & \num{ 249.4} & \A{ -798} & \num{ 120.2} \\
      chim8-4.24 & \A{ -840} & \num{ 222.8} & \A{ -834} & \num{ 199.5} & \A{ -862} & \num{  87.4} & \A{ -862} & \num{  91.6} & \B{ -880} & \num{ 123.4} \\
      chim8-4.25 & \A{ -880} & \num{  91.2} & \A{ -872} & \num{ 115.8} & \A{ -872} & \num{  80.5} & \A{ -858} & \num{ 136.8} & \B{ -890} & \num{  68.0} \\
      chim8-4.26 & \A{ -796} & \num{ 190.6} & \B{ -838} & \num{  66.1} & \A{ -792} & \num{ 205.4} & \A{ -826} & \num{ 174.8} & \A{ -788} & \num{ 220.9} \\
      chim8-4.27 & \A{ -834} & \num{  54.5} & \A{ -844} & \num{  36.3} & \B{ -846} & \num{  16.0} & \A{ -834} & \num{  57.0} & \A{ -834} & \num{  57.2} \\
      chim8-4.28 & \A{ -782} & \num{  82.4} & \A{ -746} & \num{ 249.8} & \A{ -792} & \num{  34.2} & \A{ -790} & \num{  44.2} & \B{ -798} & \num{  11.4} \\
      chim8-4.29 & \A{ -790} & \num{ 193.3} & \A{ -820} & \num{  74.4} & \A{ -792} & \num{ 186.2} & \B{ -834} & \num{  94.3} & \A{ -810} & \num{ 113.8} \\
      chim8-4.30 & \A{ -842} & \num{ 200.1} & \A{ -808} & \num{ 341.0} & \A{ -870} & \num{  87.4} & \A{ -864} & \num{ 111.5} & \B{ -890} & \num{  42.6} \\
      chim8-4.31 & \A{ -870} & \num{ 137.1} & \A{ -890} & \num{  58.8} & \B{ -900} & \num{  63.3} & \A{ -882} & \num{ 198.9} & \A{ -860} & \num{ 243.8} \\
      chim8-4.32 & \A{ -790} & \num{ 192.8} & \B{ -830} & \num{  11.1} & \A{ -824} & \num{ 144.3} & \A{ -818} & \num{ 131.9} & \A{ -822} & \num{ 160.7} \\
      chim8-4.33 & \A{ -884} & \num{  98.4} & \A{ -830} & \num{ 275.1} & \A{ -878} & \num{  83.5} & \B{ -896} & \num{  51.4} & \A{ -870} & \num{ 110.0} \\
      chim8-4.34 & \B{ -890} & \num{  24.9} & \A{ -882} & \num{  53.5} & \A{ -872} & \num{  80.2} & \A{ -876} & \num{  68.7} & \A{ -860} & \num{ 126.7} \\
      chim8-4.35 & \A{ -782} & \num{  98.6} & \B{ -798} & \num{  84.0} & \A{ -788} & \num{  51.3} & \A{ -774} & \num{ 112.7} & \A{ -794} & \num{  25.1} \\
      chim8-4.36 & \A{ -788} & \num{ 133.1} & \B{ -816} & \num{  13.4} & \A{ -792} & \num{ 111.2} & \A{ -776} & \num{ 180.2} & \A{ -780} & \num{ 164.1} \\
      chim8-4.37 & \A{ -790} & \num{  64.7} & \B{ -798} & \num{  12.1} & \A{ -760} & \num{ 176.5} & \B{ -798} & \num{  72.7} & \B{ -798} & \num{  33.7} \\
      chim8-4.38 & \A{ -806} & \num{ 349.6} & \A{ -796} & \num{ 266.9} & \B{ -856} & \num{ 154.8} & \A{ -792} & \num{ 277.7} & \A{ -840} & \num{ 184.2} \\
      chim8-4.39 & \A{ -856} & \num{  79.6} & \B{ -866} & \num{  20.7} & \A{ -850} & \num{  70.5} & \A{ -846} & \num{  86.8} & \A{ -846} & \num{  87.6} \\
      chim8-4.40 & \A{ -790} & \num{  82.9} & \A{ -760} & \num{ 303.5} & \A{ -788} & \num{  79.8} & \B{ -802} & \num{  71.6} & \A{ -800} & \num{  38.9} \\
      chim8-4.41 & \A{ -880} & \num{ 209.0} & \B{ -890} & \num{  36.7} & \A{ -846} & \num{ 185.6} & \A{ -836} & \num{ 223.6} & \A{ -864} & \num{ 114.4} \\
      chim8-4.42 & \A{ -658} & \num{ 190.7} & \B{ -694} & \num{ 210.2} & \A{ -678} & \num{  89.6} & \A{ -678} & \num{  89.9} & \A{ -684} & \num{  58.8} \\
      chim8-4.43 & \A{ -734} & \num{ 108.1} & \A{ -740} & \num{  86.0} & \B{ -756} & \num{   4.6} & \A{ -742} & \num{  70.2} & \A{ -752} & \num{  23.9} \\
      chim8-4.44 & \A{ -742} & \num{ 145.7} & \A{ -704} & \num{ 328.5} & \B{ -772} & \num{ 174.5} & \A{ -756} & \num{  81.3} & \A{ -764} & \num{  48.2} \\
      chim8-4.45 & \A{ -818} & \num{ 251.7} & \B{ -846} & \num{  26.3} & \A{ -842} & \num{  29.5} & \A{ -842} & \num{  24.6} & \A{ -834} & \num{  81.0} \\
      chim8-4.46 & \A{ -854} & \num{ 154.6} & \A{ -854} & \num{ 128.0} & \A{ -840} & \num{ 152.5} & \B{ -874} & \num{  68.8} & \A{ -832} & \num{ 179.0} \\
      chim8-4.47 & \A{ -848} & \num{ 136.6} & \A{ -856} & \num{  66.1} & \B{ -868} & \num{  96.4} & \A{ -834} & \num{ 145.0} & \A{ -850} & \num{  88.2} \\
      chim8-4.48 & \A{ -826} & \num{  94.0} & \B{ -834} & \num{  12.2} & \A{ -812} & \num{  98.6} & \A{ -812} & \num{  98.5} & \A{ -832} & \num{  18.1} \\
      chim8-4.49 & \A{ -800} & \num{ 159.8} & \B{ -820} & \num{  15.1} & \A{ -760} & \num{ 268.9} & \A{ -746} & \num{ 328.3} & \A{ -786} & \num{ 154.9} \\
      chim8-4.50 & \A{ -822} & \num{ 114.2} & \A{ -802} & \num{ 189.1} & \A{ -814} & \num{ 124.9} & \B{ -842} & \num{  72.1} & \A{ -838} & \num{  39.9} \\
      \midrule
      \#best     &         1 &              &        19 &              &        11 &              &        13 &              &         9 &              \\
      AM prim-$\int$ (\#50) &           &       167.7  &           &       118.3  &           &       119.8  &           &       130.6  &           &       109.5  \\
      GM prim-$\int$ (\#50) &           &       148.0  &           &        73.5  &           &        94.6  &           &       100.0  &           &        79.4  \\
      \bottomrule
    \end{tabular*}
  \end{scriptsize}
\end{table}

Table~\ref{tab:QUBO} shows the best primal values of geometric scaling and
the version that also uses the heuristic based on geometric scaling on the
QUBO testset. For these instances, the other stand-alone augmentation methods do not perform
well -- and we skip their results here.

The three variants \textsf{geom-heur}, \textsf{geom-heur-infer}, and
\textsf{geom-heur-64} find significantly better primal solutions than the
default settings. Moreover, their primal integral is significantly
smaller. Among the three variants, \textsf{geom-heur-infer} performs
slightly better than the other two. However, all three variants find the
best solutions for some instances for which all other variants are not as
good. As for the other
two testsets, the geometric scaling heuristic uses more phases than
\textsf{geom-64}, on average (\textsf{geom-64}: 3.2, \textsf{geom-heur}:
19.2, \textsf{geom-heur-64}: 4.8). The good performance comes from the fact that usually the other
heuristics find good solutions, which can then easily be improved to even
better solutions by the geometric scaling heuristic.

In any case, these results are surpassed by \textsf{geom-64}, which finds
the largest number of best solutions and produces the smallest primal
integral. These excellent results arise from the fact that very few phases
are needed in order to arrive at a level of \(\mu\) that helps to improve
the primal solutions. Indeed, it is often the case that for a particular
\(\mu\) a series of improving solutions is found until the time limit is
reached.

In summary, the QUBO instances show the excellent potential of geometric
scaling. It is likely that extensive parameter tuning could help to even
improve these results.

\section*{Acknowledgements}
\label{sec:acknowledgements}

Research reported in this paper was partially supported by the NSF
grants CMMI-1300144 and CCF-1415460, as well as the AFOSR grant
FA9550-12-1-0151. We would like to thank Sanjeeb Dash for providing us
with the QUBO instances and valuable insights.

\bibliographystyle{abbrvnat}
\bibliography{bibliography}

\begin{thebibliography}{40}
\providecommand{\natexlab}[1]{#1}
\providecommand{\url}[1]{\texttt{#1}}
\expandafter\ifx\csname urlstyle\endcsname\relax
  \providecommand{\doi}[1]{doi: #1}\else
  \providecommand{\doi}{doi: \begingroup \urlstyle{rm}\Url}\fi

\bibitem[Achterberg(2009)]{Ach2009}
T.~Achterberg.
\newblock {SCIP}: Solving constraint integer programs.
\newblock \emph{Mathematical Programming Computation}, 1\penalty0 (1):\penalty0
  1--41, 2009.

\bibitem[Arora et~al.(2012)Arora, Hazan, and Kale]{arora2012multiplicative}
S.~Arora, E.~Hazan, and S.~Kale.
\newblock The multiplicative weights update method: a meta-algorithm and
  applications.
\newblock \emph{Theory of Computing}, 8\penalty0 (1):\penalty0 121--164, 2012.

\bibitem[Ben-Tal and Nemirovski(2001)]{bental2001lectures}
A.~Ben-Tal and A.~Nemirovski.
\newblock \emph{Lectures on modern convex optimization: analysis, algorithms,
  and engineering applications}.
\newblock Society for Industrial and Applied Mathematics, Philadelphia, PA,
  USA, 2001.

\bibitem[Berthold(2013)]{Berthold13}
T.~Berthold.
\newblock Measuring the impact of primal heuristics.
\newblock \emph{Operations Research Letters}, 41\penalty0 (6):\penalty0
  611--614, 2013.

\bibitem[Berthold(2014)]{Berthold14}
T.~Berthold.
\newblock \emph{Heuristic algorithms in global MINLP solvers}.
\newblock PhD thesis, TU Berlin, 2014.

\bibitem[Bienstock(1999)]{bienstock1999approximately}
D.~Bienstock.
\newblock Approximately solving large-scale linear programs. {I.}
  {S}trengthening lower bounds and accelerating convergence.
\newblock \emph{CORC Report}, 1999.

\bibitem[Bienstock(2002)]{bienstock2002potential}
D.~Bienstock.
\newblock \emph{Potential Function Methods for Approximately Solving Linear
  Programming Problems: Theory and Practice}, volume~53 of \emph{International
  Series in Operations Research \& Management Science}.
\newblock Springer, 2002.

\bibitem[Conn et~al.(2000)Conn, Gould, and Toint]{ConnGouldToint2000}
A.~R. Conn, N.~I.~M. Gould, and P.~L. Toint.
\newblock \emph{Trust Region Methods}.
\newblock SIAM, 2000.

\bibitem[Dash(2013)]{dash2013note}
S.~Dash.
\newblock A note on {QUBO} instances defined on {C}himera graphs.
\newblock \emph{preprint arXiv:1306.1202}, 2013.

\bibitem[Dash and Puget(2015)]{dash2015optima}
S.~Dash and J.-F. Puget.
\newblock On quadratic unconstrained binary optimization problems defined on
  {C}himera graphs.
\newblock \emph{Optima}, 98:\penalty0 2--6, 2015.

\bibitem[{De Loera} et~al.(2008){De Loera}, Hemmecke, K{\"{o}}ppe, and
  Weismantel]{LoeHKW08}
J.~A. {De Loera}, R.~Hemmecke, M.~K{\"{o}}ppe, and R.~Weismantel.
\newblock {FPTAS} for optimizing polynomials over the mixed-integer points of
  polytopes in fixed dimension.
\newblock \emph{Math. Program.}, 115\penalty0 (2):\penalty0 273--290, 2008.

\bibitem[De~Loera et~al.(2013)De~Loera, Hemmecke, and
  K{\"o}ppe]{de2013algebraic}
J.~A. De~Loera, R.~Hemmecke, and M.~K{\"o}ppe.
\newblock \emph{Algebraic and geometric ideas in the theory of discrete
  optimization}, volume~14 of \emph{MOS-SIAM Series on Optimization}.
\newblock SIAM, 2013.

\bibitem[De~Loera et~al.(2014)De~Loera, Hemmecke, and Lee]{de2014augmentation}
J.~A. De~Loera, R.~Hemmecke, and J.~Lee.
\newblock Augmentation in linear and integer linear programming.
\newblock \emph{Preprint arXiv:1408.3518}, 2014.

\bibitem[Edmonds and Karp(1972)]{edmonds1972theoretical}
J.~Edmonds and R.~M. Karp.
\newblock Theoretical improvements in algorithmic efficiency for network flow
  problems.
\newblock \emph{Journal of the ACM}, 19\penalty0 (2):\penalty0 248--264, 1972.

\bibitem[Fischetti and Lodi(2003)]{FisL03}
M.~Fischetti and A.~Lodi.
\newblock Local branching.
\newblock \emph{Mathematical Programming}, 98\penalty0 (1-3):\penalty0 23--47,
  2003.

\bibitem[Fischetti and Monaci(2014)]{fischetti2014proximity}
M.~Fischetti and M.~Monaci.
\newblock Proximity search for 0-1 mixed-integer convex programming.
\newblock \emph{Journal of Heuristics}, 20\penalty0 (6):\penalty0 709--731,
  2014.

\bibitem[Frank and Tardos(1987)]{frank1987application}
A.~Frank and {\'E}.~Tardos.
\newblock An application of simultaneous {D}iophantine approximation in
  combinatorial optimization.
\newblock \emph{Combinatorica}, 7\penalty0 (1):\penalty0 49--65, 1987.

\bibitem[Garg and Koenemann(2007)]{garg2007faster}
N.~Garg and J.~Koenemann.
\newblock Faster and simpler algorithms for multicommodity flow and other
  fractional packing problems.
\newblock \emph{SIAM Journal on Computing}, 37\penalty0 (2):\penalty0 630--652,
  2007.

\bibitem[Graham et~al.(1995)Graham, Gr{\"o}tschel, and
  Lov{\'a}sz]{graham1995handbook}
R.~L. Graham, M.~Gr{\"o}tschel, and L.~Lov{\'a}sz.
\newblock \emph{Handbook of combinatorics}, volume~1.
\newblock Elsevier, 1995.

\bibitem[Graver(1975)]{Gra75}
J.~E. Graver.
\newblock On the foundations of linear and integer linear programming i.
\newblock \emph{Mathematical Programming}, 9\penalty0 (1):\penalty0 207--226,
  1975.

\bibitem[Hansen et~al.(2006)Hansen, Mladenovi\'{c}, and
  Uro\v{s}evi\'{c}]{HanMU06}
P.~Hansen, N.~Mladenovi\'{c}, and D.~Uro\v{s}evi\'{c}.
\newblock Variable neighborhood search and local branching.
\newblock \emph{Computers and Operations Research}, 33:\penalty0 3034--3045,
  2006.

\bibitem[Hemmecke et~al.(2010)Hemmecke, K{\"{o}}ppe, Lee, and
  Weismantel]{HemLW10}
R.~Hemmecke, M.~K{\"{o}}ppe, J.~Lee, and R.~Weismantel.
\newblock Nonlinear integer programming.
\newblock In M.~J{\"{u}}nger, T.~M. Liebling, D.~Naddef, G.~L. Nemhauser, W.~R.
  Pulleyblank, G.~Reinelt, G.~Rinaldi, and L.~A. Wolsey, editors, \emph{50
  Years of Integer Programming 1958--2008 -- From the Early Years to the
  State-of-the-Art}, pages 561--618. Springer, 2010.

\bibitem[Hemmecke et~al.(2011)Hemmecke, Onn, and Weismantel]{HemOW11}
R.~Hemmecke, S.~Onn, and R.~Weismantel.
\newblock A polynomial oracle-time algorithm for convex integer minimization.
\newblock \emph{Math. Program.}, 126\penalty0 (1):\penalty0 97--117, 2011.

\bibitem[Hemmecke et~al.(2014)Hemmecke, Köppe, and Weismantel]{HemKW14}
R.~Hemmecke, M.~Köppe, and R.~Weismantel.
\newblock Graver basis and proximity techniques for block-structured separable
  convex integer minimization problems.
\newblock \emph{Mathematical Programming}, 145\penalty0 (1-2):\penalty0 1--18,
  2014.

\bibitem[Koch et~al.(2011)Koch, Achterberg, Andersen, Bastert, Berthold, Bixby,
  Danna, Gamrath, Gleixner, Heinz, Lodi, Mittelmann, Ralphs, Salvagnin, Steffy,
  and Wolter]{Kochetal2011}
T.~Koch, T.~Achterberg, E.~Andersen, O.~Bastert, T.~Berthold, R.~E. Bixby,
  E.~Danna, G.~Gamrath, A.~M. Gleixner, S.~Heinz, A.~Lodi, H.~Mittelmann,
  T.~Ralphs, D.~Salvagnin, D.~E. Steffy, and K.~Wolter.
\newblock M{IPLIB} 2010: mixed integer programming library version 5.
\newblock \emph{Math. Program. Comput.}, 3\penalty0 (2):\penalty0 103--163,
  2011.

\bibitem[Lee et~al.(2008)Lee, Onn, and Weismantel]{LeeOW08}
J.~Lee, S.~Onn, and R.~Weismantel.
\newblock On test sets for nonlinear integer maximization.
\newblock \emph{Oper. Res. Lett.}, 36\penalty0 (4):\penalty0 439--443, 2008.

\bibitem[Lee et~al.(2012)Lee, Onn, Romanchuk, and Weismantel]{LeeORW12}
J.~Lee, S.~Onn, L.~Romanchuk, and R.~Weismantel.
\newblock The quadratic {G}raver cone, quadratic integer minimization, and
  extensions.
\newblock \emph{Math. Program.}, 136\penalty0 (2):\penalty0 301--323, 2012.

\bibitem[Letchford and Lodi(2003)]{letchford2003augment}
A.~N. Letchford and A.~Lodi.
\newblock An augment-and-branch-and-cut framework for mixed 0-1 programming.
\newblock In \emph{Combinatorial Optimization---Eureka, You Shrink!}, pages
  119--133. Springer, 2003.

\bibitem[McCormick and Shioura(2000)]{mccormick2000minimum}
S.~T. McCormick and A.~Shioura.
\newblock Minimum ratio canceling is oracle polynomial for linear programming,
  but not strongly polynomial, even for networks.
\newblock \emph{Operations Research Letters}, 27\penalty0 (5):\penalty0
  199--207, 2000.

\bibitem[Nemhauser and Wolsey(1988)]{nw1988}
G.~Nemhauser and L.~Wolsey.
\newblock \emph{Integer and Combinatorial Optimization}.
\newblock Wiley, 1988.

\bibitem[Onn(2010)]{Onn10}
S.~Onn.
\newblock \emph{Nonlinear Discrete Optimization: An Algorithmic Theory}.
\newblock Zurich lectures in advanced mathematics. European Mathematical
  Society Publishing House, 2010.

\bibitem[Orlin and Ahuja(1992)]{orlin1992new}
J.~B. Orlin and R.~K. Ahuja.
\newblock New scaling algorithms for the assignment and minimum mean cycle
  problems.
\newblock \emph{Mathematical Programming}, 54\penalty0 (1-3):\penalty0 41--56,
  1992.

\bibitem[Plotkin et~al.(1995)Plotkin, Shmoys, and Tardos]{plotkin1995fast}
S.~A. Plotkin, D.~B. Shmoys, and {\'E}.~Tardos.
\newblock Fast approximation algorithms for fractional packing and covering
  problems.
\newblock \emph{Mathematics of Operations Research}, 20\penalty0 (2):\penalty0
  257--301, 1995.

\bibitem[Rockafellar(1976)]{Rockafellar1976}
R.~T. Rockafellar.
\newblock Monotone operators and the proximal point algorithm.
\newblock \emph{SIAM Journal on Control and Optimization}, 14\penalty0
  (5):\penalty0 877--898, 1976.

\bibitem[Scarf(1997)]{Sca97}
H.~E. Scarf.
\newblock Test sets for integer programs.
\newblock \emph{Mathematical Programming}, 79\penalty0 (1-3):\penalty0
  355--368, 1997.

\bibitem[Schrijver(1986)]{schrijver1986theory}
A.~Schrijver.
\newblock \emph{{Theory of linear and integer programming}}.
\newblock Wiley, 1986.

\bibitem[Schulz and Weismantel(2002)]{schulz2002complexity}
A.~S. Schulz and R.~Weismantel.
\newblock The complexity of generic primal algorithms for solving general
  integer programs.
\newblock \emph{Mathematics of Operations Research}, 27\penalty0 (4):\penalty0
  681--692, 2002.

\bibitem[Schulz et~al.(1995)Schulz, Weismantel, and Ziegler]{schulz19950}
A.~S. Schulz, R.~Weismantel, and G.~M. Ziegler.
\newblock 0/1-integer programming: Optimization and augmentation are
  equivalent.
\newblock In \emph{Algorithms -- {ESA} '95, Proceedings}, pages 473--483, 1995.

\bibitem[SCIP()]{SCIP}
SCIP.
\newblock {S}olving {C}onstraint {I}nteger {P}rograms, {V}ersion 3.2.0, 2015.
\newblock \url{http://scip.zib.de/}.

\bibitem[Wallacher and Zimmermann(1992)]{wallacher1992combinatorial}
C.~Wallacher and U.~Zimmermann.
\newblock A combinatorial interior point method for network flow problems.
\newblock \emph{Mathematical programming}, 56\penalty0 (1-3):\penalty0
  321--335, 1992.

\end{thebibliography}

\includepdf[pages=-]{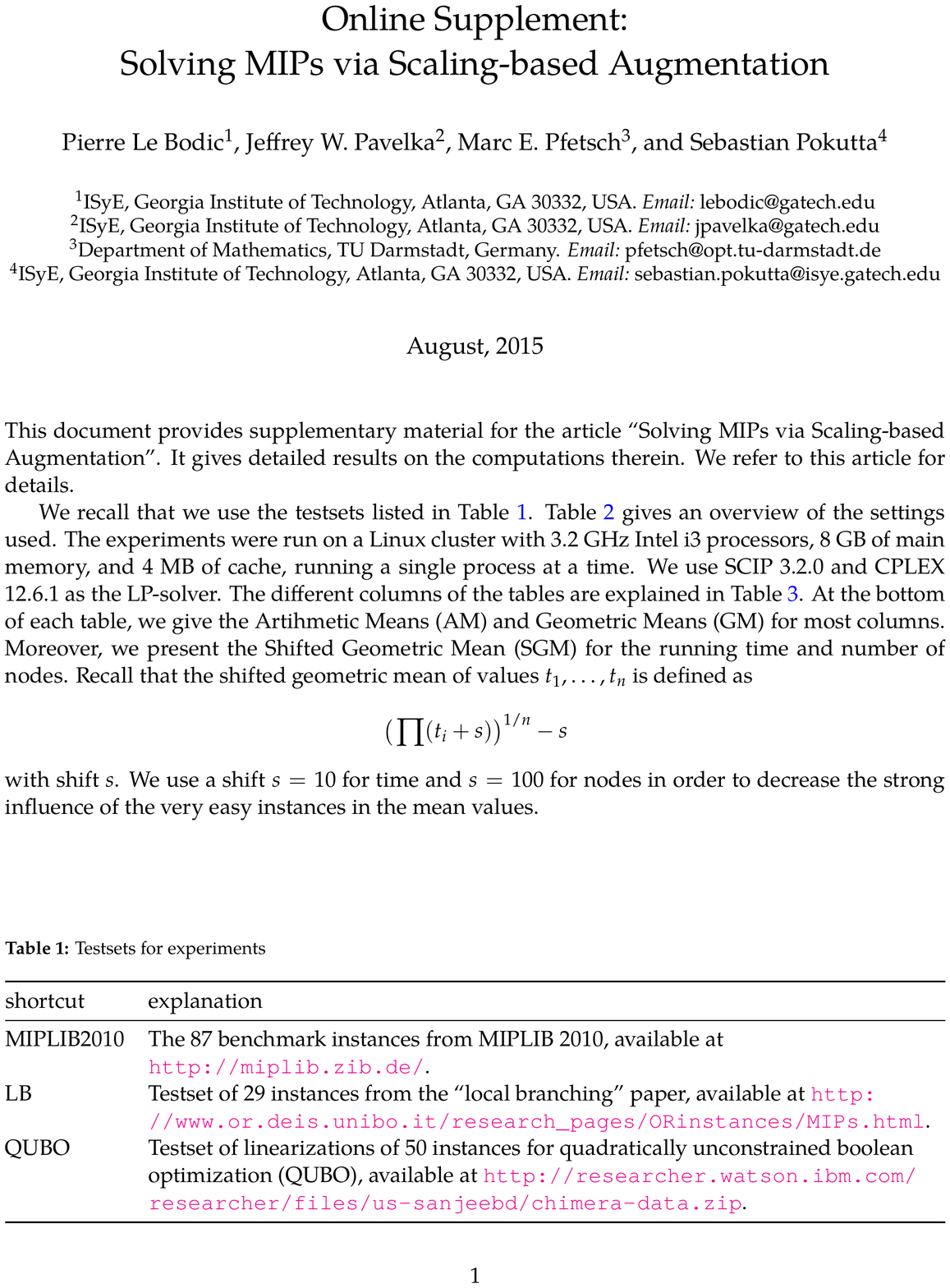}

\end{document}